\documentclass[leqno,12pt]{article}

\usepackage[usenames, dvipsnames]{xcolor}
\usepackage{pgf,tikz} % Pour les figures avec tikz\usetikzlibrary{shapes,decorations} %pour rajouter des figures Ã  tikz : diamond, etc...
\usetikzlibrary{shapes,arrows,chains}
\usepackage{setspace} 
\usepackage{mathpazo}
\usepackage{amssymb,hyperref,amsthm}
\usepackage{euscript}
\usepackage{flafter}
\usepackage{pstricks}
\usepackage[latin1]{inputenc}
\usepackage{amsmath}
\usepackage{amsfonts}
\usepackage{pstricks-add}
\usepackage{yfonts}
\usepackage{egothic}
\usepackage{dsfont}
\usepackage{bm}
\usepackage{variations}
\usepackage{mathtools}
\usepackage{mathrsfs}
\usepackage{graphicx}   % for \includegraphics and trim/clip options
\usepackage{float}      % for the [H] placement specifier
\usepackage{subcaption} % for the subfigure environment

\usepackage[most]{tcolorbox}

\hypersetup{
    linktoc=page,
   linkcolor=red,          % color of internal links
  citecolor=blue,        % color of links to bibliography
    filecolor=blue,      % color of file links
   urlcolor=cyan,
    colorlinks=true           % color of external links
}

\usepackage[refpage]{nomencl}
\usepackage{makeidx}
\makeindex

\evensidemargin\oddsidemargin
\newcommand{\eqnsection}{
\renewcommand{\theequation}{\thesection.\arabic{equation}}
   \makeatletter
   \csname  @addtoreset\endcsname{equation}{section}
   \makeatother}
\eqnsection
\def\R{{\mathbb R}}
\def\e{{\mathbb E}}
\def\p{{\mathbb P}}
\def\P{{\bf P}}
\def\E{{\bf E}}
\def\Q{{\bf Q}}
\def\Z{{\mathbb Z}}
\def\N{{\mathbb N}}
\def\T{{\mathbb T}}
\def\S{\mathcal{S}}
\newtheorem{thm}{Theorem}[section]
\newtheorem{prop}[thm]{Proposition}
\newtheorem{lem}[thm]{Lemma}

\newtheorem{condition}{Assumption}

\newcommand{\floor}[1]{{\left\lfloor #1 \right\rfloor}}
\newcommand{\ceil}[1]{{\left\lceil #1 \right\rceil}}
\newcommand{\ind}[1]{\mathbf{1}_{\left\{ #1 \right\}}}
\renewcommand{\epsilon}{\varepsilon}

\newcommand{\egloi}{\stackrel{\textrm{(d)}}{=}}

\usepackage{stmaryrd} % Pour utiliser les doubles crochets des intervalles d'entiers \llbracket \rrbracket
\title{\bf Branching random walk in random environment%Tail behaviours of branching random walks, of associated multitype Galton-Watson tree and of branching process with immigration in random environment
}

\author{Xinxin Chen\thanks{School of Mathematical Sciences, Beijing Normal University, xinxin.chen@bnu.edu.cn}, Chenlin Gu\thanks{Yau Mathematical Sciences Center, Tsinghua University, gclmath@tsinghua.edu.cn} and Zhiqi Zhao\thanks{Department of Statistics, University of Oxford, zhiqi.zhao@stats.ox.ac.uk}}

\begin{document}

\maketitle
%\listoftodos
\begin{abstract}
We consider a branching random walk on \(\Z^d\) in a random environment given by Bernoulli site percolation with parameter \(p\in (0,1)\). In this model, each particle located at an open site reproduces according to a law \(\mu_\circ\), whereas a particle at a closed site reproduces according to another law \(\mu_\bullet\). Each newly born child performs an independent simple random walk jump from the position of its parent. We study the quenched survival probability under various assumptions on \((\mu_\circ, \mu_\bullet)\) and establish a Yaglom theorem when both offspring distributions are critical.
%We consider a branching random walk in $\Z^d$ where the random environment is given by Bernoulli site percolation with parameter $p\in (0,1)$. In this model, every particle at open site, branches according to the law $\mu_\circ$ while every particle at closed site branches according to another law $\mu_\bullet$. Each new born child makes an independent simple random walk jump from the position of its parent. We study the quenched survival probability under different conditions on $(\mu_\circ, \mu_\bullet)$ and establish a Yaglom theorem when both offspring laws are critical. 

%The environment is given by a branching random walk for which the tail distribution of $W_\infty$, the limit of its additive martingale has been given in \cite{Liu}. In this work, we study the joint tail distribution of $W_\infty$ and the minimum of the branching random walk, and then use it to obtain the tail behaviour of the largest local times of the associated random walk in random environment during one excursion. We also establish the asymptotic behaviour of the largest local times during $n$ excursions. %Our arguments also show the tail distribution of the maximum of a branching process in random environment with immigration.


\end{abstract}

\section{Introduction: Models and results}
\label{Intro}
\subsection{Branching random walk in random environment}
Let us consider a branching random walk in i.i.d. random environment. First, take $p\in(0,1)$ and sample a Bernoulli site percolation with parameter $p$ on $\Z^d$. Denote it by $\omega=\{\omega(x)\}_{x\in \Z^d}$ whose law is $\P_p$. This means that under $\P_p$, $\{\omega(x), x\in\Z^d\}$ are i.i.d. Bernoulli random variables. A site $x$ is said to be open if $\omega(x) =1$, otherwise it is closed. This $\omega$ serves as environment in our branching model. 

Let $\mu_\circ$ and $\mu_\bullet$ be two probability measures on $\N=\{0,1,2,\cdots\}$. Let $\nu$ be a probability measure on $\Z^d$ which will be the motion law. 

Given $\omega$, we run a branching random walk started from $x\in\Z^d$ in the following way. At time $0$, there is an ancestor denoted by $\emptyset$, located at $S_\emptyset = x$. At time $1$, the root particle $\emptyset$ is replaced by $N_\emptyset$ children where $N_\emptyset$ has the law $\mu_\circ$ if $\omega(S_\emptyset)=1$ and it has the law $\mu_\bullet$ if $\omega(S_{\emptyset})=0$. If $N_\emptyset =0$, the system stops. Otherwise, we label the children by $1,2,\cdots, N_\emptyset$ and each child $i$ makes an independent jump to $S_i$ according to $\nu$. We thus get the particles at the first generation and their positions: $\{(i,S_i), 1\le i\le N_\emptyset\}$. At time 2, every particle of the first generation dies and independently gives birth to its own children and the offspring of $i$ is labelled  by $i1, i2, \cdots, iN_i$. If $\omega(S_i)=1$, $N_i$ follows the law $\mu_\circ$ and it follows the law $\mu_\bullet$ if $\omega(S_i)=0$. Each new born child $ij$ makes an independent jump from $S_i$ to $S_{ij}$ according to $\nu$, for $1\le j\le N_i, 1\le i \le N_\emptyset$. The system continues to evolve in this way and we get a branching random walk in random environment.

We use $\p^\omega_x$ and $\e^\omega_x$ to denote the quenched law and the corresponding expectation if the system starts from $S_\emptyset=x$, and use $\P_p$ and $\E_p$ to denote the environment law and its expectation. And let $\p_x= \P_p\otimes \p^\omega_x$ and $\e_x$ be the annealed probability and annealed expectation. When $x=0$, we use $\p^\omega$ and $\p$ for abbreviation.

Let \(\T\) denote the genealogical tree of the process. Then \(\T\subset \mathbb{U}:=\{\emptyset\}\cup\bigcup_{k\ge1}\N_+^k\), where \(\N_+:=\N\setminus\{0\}\), and \(\T\) is a random tree. For any particle \(u\in\T\), we set \(|u|=0\) if \(u=\emptyset\), and \(|u|=k\) if \(u=u_1\cdots u_k\in\N_+^k\) for some \(k\ge1\); its position is denoted by \(S_u\). Define
\begin{equation}
Z_n(\cdot):=\sum_{|u|=n}\delta_{S_u}(\cdot), \qquad n\ge0,
\end{equation}
the counting measure at time \(n\). In particular, we write \(Z_n=Z_n(\Z^d)\) for the total population size at generation \(n\).

In this work, we first study the quenched survival event
\[
\S:=\{Z_n\ge1 \text{ for all } n\ge0\}.
\]
Second, in the special case where both \(\mu_\circ\) and \(\mu_\bullet\) are critical, we establish the Kolmogorov estimate and the Yaglom theorem under the quenched probability.

For simplicity, throughout this paper we assume that the motion law is
\begin{equation}
\nu = \sum_{|e|_1=1}\frac{1}{2d}\delta_e,
\end{equation}
which corresponds to simple random walk jumps, where \(|\cdot|_1\) denotes the \(\ell^1\)-norm.

%Let $\T$ be the genealogical tree of this model. Then $\T\subset \mathbb{U}:=\{\emptyset\}\bigcup\cup_{k\ge1}\N_+^k$ with $\N_+:=\N\setminus\{0\}$, and it is a random tree. For every particle born in this model $u\in\T$, if $u=\emptyset$, we set $|u|=0$; if $u = u_1\cdots u_k \in \N_+^k$ for some $k\ge 1$, we use $|u|=k$ to denote its generation and $S_u$ its position. Let
%\begin{equation}
%Z_n(\cdot):=\sum_{|u|=n}\delta_{S_u}(\cdot), \forall n\ge 0,
%\end{equation}
%be the counting measure at time $n$. In particular, write $Z_n = Z_n(\Z^d)$.
 
%In this work, we first study\[\S:=\{Z_n\ge1, \forall n\ge0\}\]
%under the quenched $\p^\omega$. Secondly,  in the particular case where both $\mu_\circ$ and $\mu_\bullet$ are critical, we are going to establish the Kolmogorov estimate and Yaglom theorem.

%For simplicity, throughout this paper, we always assume that the motion law is
%\begin{equation}
%\nu = \sum_{|e|_1=1}\frac{1}{2d}\delta_e
%\end{equation} 
%which corresponds to the simple random walk jump with $|\cdot|_1$ being $\ell^1$-norm.

To place our work in a broader perspective, we briefly recall several related models of branching random walks in spatially random environments. Comets and Popov~\cite{CP2007} investigated supercritical branching random walks in an i.i.d. random environment on \(\mathbb Z^d\), where both the branching mechanism and the jump law depend on the spatial site. They established a recurrence/transience dichotomy and proved a shape theorem in the recurrent regime. Den Hollander, Menshikov and Popov~\cite{HMP1999} analyzed a model in which branching is allowed only at a sparse random set of sites, and obtained sufficient conditions for recurrence and transience. In dimension \(d=1\), criterion for recurrence/transience was investigated by Gantert et al \cite{GMPV2010} for i.i.d. environment and was also considered by Machado and Popov \cite{MP2000} for Markovian random environment. When the space is regular tree, Machado and Popov~\cite{MP2003} provided criteria for recurrence and transience. Moreover, Devulder~\cite{Dev2007} studied the velocity of the rightmost particle for the one-dimensional model with constant branching rate and a left-biased random jump law.

In the continuous-space setting, for critical branching Brownian motion with soft killing in Poissionnian obstacles, Le Gall and V\'eber \cite{LGV2012} studied the quenched probability of exiting a large ball. In the supercritical case, Engl\"ander~\cite{Eng2008} considered a dyadic branching Brownian motion whose branching rate is suppressed by mild Poissonian obstacles, and established a quenched law of large numbers for the total population size; the question of the decay rate of the survival probability in the critical case was also raised there. In the discrete setting considered in the present paper, where Brownian motion is replaced by a random walk and Poissonian obstacles by an i.i.d. environment, Engl\"ander and Siebern~\cite{ES} proved almost sure extinction in the critical case, and for \(d=1\), formulated a conjecture, supported by simulations, on the decay rate. Subsequently, Engl\"ander and Peres~\cite{EP} derived quenched Kolmogorov-type estimates in the case where \(\mu_\bullet=\delta_1\) and \(\mu_\circ=\frac12\delta_0+\frac12\delta_2\), as well as in the subcritical regime. They also conjectured the quenched Yaglom theorem in the critical regime which will be verified in the current paper.

\subsection{Main results}

Denote the mean numbers of offsprings by
\[
m_\circ = \sum_{k\ge0} k\mu_\circ(k),\qquad m_\bullet = \sum_{k\ge0} k\mu_\bullet(k).
\]
Observe that offspring law $\delta_1$ means degenerately having $1$ child and that $\delta_0$ means pure killing. 

Our first theorem discusses the quenched survival probability as follows. 

\begin{thm}\label{thm: Sprobab}
The following holds for any dimension and for any $p\in(0,1)$.
\begin{enumerate}
\item[\textbf{(1)}] If $\max\{m_\circ, m_\bullet\}\le 1$ and $\min\{\mu_\circ(1), \mu_\bullet(1) \}<1$, then for $\P_p$-a.s. environment $\omega$,
\[
\p^\omega(\S) =0.
\]
\item[\textbf{(2)}] If $\max\{m_\circ, m_\bullet\}> 1$ and $\max\{\mu_\circ(0), \mu_\bullet(0) \}<1$, then for $\P_p$-a.s. environment $\omega$,
\[
\p^\omega(\S) >0.
\]
\item[\textbf{(3)}] If $\max\{m_\circ, m_\bullet\}> 1$ and $\max\{\mu_\circ(0), \mu_\bullet(0) \}=1$, then 
\[
\p(\S)>0,
\]
and 
\[
\P_p\left( \omega : \p^\omega(\S) = 0 \right) >0. 
\]
\end{enumerate}

\end{thm}

Let us work more for the third case. For convenience, we could suppose that $\mu_\bullet = \delta_0$ and $m_\circ >1$. Then particles get killed when exiting the open cluster of the origin $\mathcal{C}(0)$. Let us recall some basic facts on Bernoulli site percolation here. Set
\[
\theta_{\Z^d}(p):= \P_p(\#\mathcal{C}(0) = \infty ). 
\]
Then $\theta_{\Z^d}(p)\equiv 0$ for $d=1$. When $d\ge 2$, a phase transition happens (see \cite{Grimmett99} ) and there exists $p_c({\Z^d})\in(0,1)$ such that
\[
\theta_{\Z^d}(p) \begin{cases}
>0, \textrm{ if } p\in (p_c({\Z^d}) , 1)\\
=0, \textrm{ if } p\in (0, p_c({\Z^d})).
\end{cases}
\]
A major conjecture states that $\theta_{\Z^d}(p_c) = 0$ for all $d \geq 2$, but it is only proved for $d = 2$ (\cite{K80}) and $d \geq 11$ (\cite{HS90, FH17}). We also refer to the proceeding \cite{DC18} for more discussions.

When $d=1$, $\mathcal{C}(0)$ is a finite interval and $\p^\omega(\S)$ is the survival probability of a branching random walk in a strip. For a branching Brownian motion in a strip, this problem is studied by Harris-Hesse-Kyprianou \cite{HHK} and it turns out that $\p^\omega(\S)>0$ if and only if the length of the strip exceeds some threshold (depending on the offspring law). In higher dimensions, one can refer to Engl\"ander-Kyprinaou \cite{EK} which proves a dichotomy on local extinction and  local exponential growth for spatial branching processes in $D\subset\R^d$. %where the generalized principal eigenvalue for $L+\beta$ on the domain $D$ in the study of the spatial $(L,\beta;D)$-branching process. %it is natural to believe that $\p^\omega(S)>0$ if $\mathcal{C}(0)$ is sufficiently large. 

In our setting, it is natural to believe that $\p^\omega(\S)>0$ if $\mathcal{C}(0)$ is sufficiently large, especially if $\#\mathcal{C}(0)=\infty$. Nevertheless, on $\{\#\mathcal{C}(0)<\infty\}$, the survival probability depends on the competition between the branching mechanism and the killing rate upon leaving the finite domain $\mathcal{C}(0)$. In the general theory, for any fixed finite and connected set $A\subset\subset \Z^d$, we introduce the transition operator $T_A$ of the simple random walk in $A$ and killed when leaving $A$, acting on functions $f:A\to \R$ as follows
\begin{equation}
(T_Af)(y):=\frac{1}{2d}\sum_{\substack{|e|_1=1\\y+e\in A}}f(y+e).
\end{equation}
Let $\rho_A$ be the spectral radius of $T_A$. In Lemma 2.1, we will show that as long as $\#A\ge2$, $\rho_A\in(0,1)$. For singleton and empty set, set $\rho_{\{x\}}=\rho_\emptyset=0$. Further, we will see that the system has positive survival probability $\p^\omega(\S)>0$ on $\{\#\mathcal{C}(0)<\infty\}$ only if $\rho_{\mathcal{C}(0)} m_\circ >1$ with the convention that $0\cdot \infty =0$.

Now, let us state the second theorem, dealing with the case where  $\mu_\bullet = \delta_0$ and $m_\circ >1$.
%Moreover, we set
%\[
%\Theta(p, \mu_\circ):= \P_p\left( \omega : \p^\omega(\S) > 0 \right) .
%\]
%Immediately from Theorem \ref{thm: Sprobab} and the next theorem, we see that $\Theta(p, \mu_\circ) \in [\theta(p), p)$. 

\begin{thm}\label{thm: Sprobab+killing}
Suppose that $\mu_\bullet = \delta_0$ and $m_\circ >1$. %Let $d\ge 2$ and $p\in (p_c({\Z^d}) , 1)$. For $\P_p$-a.s. on $\{ \mathcal{C}(0) = \infty \}$,
\begin{itemize}
\item \textbf{Infinite cluster case.} If $\theta_{\Z^d}(p)>0$, then for $\P_p$-a.s. $\omega \in\{\#\mathcal{C}(0)=\infty\}$, we have
\begin{equation*}
\p^\omega(S) >0.
\end{equation*}
\item \textbf{Finite cluster case.} For any $p\in(0,1)$, $\P_p$-a.s., we have
\begin{equation*}
\{\omega : \p^\omega(S)>0\} \cap  \{\#\mathcal{C}(0)<\infty\} \overset{a.s.}{=} \{\omega:  \#\mathcal{C}(0)<\infty, \ \rho_{\mathcal{C}(0)}m_\circ >1\}.
\end{equation*}
Therefore, 
\begin{equation}
\P_p(\omega: \p^\omega(S)>0) = \theta_{\Z^d}(p) + \P_p\left( \omega:  \#\mathcal{C}(0)<\infty, \ \rho_{\mathcal{C}(0)}m_\circ >1 \right).
\end{equation}
\end{itemize}

\end{thm}

When $\p^\omega(S)=0$ for $\P_p$-a.s. $\omega$, the system dies out eventually. So, following Theorem 1.1 part-1, if $\max\{m_\circ, m_\bullet\}\le 1$ and $\min\{\mu_\circ(1), \mu_\bullet(1) \}<1$, it is natural to study the Kolmogorov estimate, i.e., the decay of the survival probabilities up to time $n$:
\begin{equation*}
\p^\omega(Z_n \ge 1) \textrm{ and } \p(Z_n \ge 1),
\end{equation*}
as well as the corresponding Yaglom theorems. 

In one special case when $\mu_\circ = \frac12\delta_0+ \frac12\delta_2$ and $\mu_\bullet=\delta_1$, Engl\"ander and Peres \cite{EP} proved that for $\P_p$-a.s. $\omega$, 
\[
\p^\omega(Z_n\ge 1) \sim \frac{2}{pn}. 
\]
And they conjecture that quenched Yaglom theorem also holds in their setting. Further, by numerical simulation, Engl\"ander and Sieben \cite{ES} conjecture that the annealed Kolmogorov estimate behaves differently from the classical model and depends on the dimension $d$. More precisely, in dimension $d=1$, Conjecture 6.1 in \cite{ES} says that there exists some positive constant $C>0$ such that as $n\to\infty$,
\begin{align}\label{eq.conjecture6.1}
    \p(Z_n \ge 1) = \frac{2}{pn } + \frac{C+o(1)}{ n^{3/2}},
\end{align}
whereas for the classical critical Galton-Watson process with offspring $\frac12\delta_0+ \frac12\delta_2$, $\p(Z_n \ge 1) = \frac{2}{n}+ O(\frac{(\log n)^2}{n^2})$, see \cite{VZ} or Lemma 2.2 of \cite{PRR}.

In another special case when $\mu_\bullet = \delta_1$ and $\mu_\circ$ is subcritical with $m_\circ\in(0,1)$, Engl\"ander and Peres \cite{EP} obtain that for there exists positive constant $C_{d,p}$ which does NOT dependent on $\mu_\circ$ such that for $\P_p$-a.s. $\omega$, 
\begin{equation}\label{sub-EP}
\p^\omega( Z_n \ge 1 ) = \exp\left( - [C_{d,p} +o(1)]\frac{n}{(\log n)^{2/d}} \right), \textrm{ as } n\to\infty.
\end{equation}
This is closely related to the random walk in random obstacles. And the Kolmogorov estimate is of totally different decaying compared with the classical subcritical Galton-Watson process. 

In the present paper, we aim to prove the Yaglom theorem conjectured by Engl\"ander and Peres~\cite{EP}, focusing on the \textit{critical} case under the following assumption:
\begin{condition}\label{A1}
\begin{enumerate}
\item \textbf{Criticality:}
\begin{equation*}
m_\circ = m_\bullet =1.
\end{equation*}
\item \textbf{Non-denegeracy:}
\begin{equation*}
\min\{\mu_\circ(1), \mu_\bullet(1) \}<1.
\end{equation*}
\item \textbf{Finite variance:}
\begin{equation*}
\sigma^2_\circ := \sum_{k\ge0} k^2\mu_\circ(k) -1 <\infty,\textrm{ and } \sigma_\bullet^2:=\sum_{k\ge0}k^2\mu_\bullet(k) -1 <\infty.
\end{equation*}
\end{enumerate}

\end{condition}

Before stating the Yaglom theorem, we first give a generalized version of Kolmogorov estimate which will be proved by use of spinal decomposition.

\begin{thm}\label{thm: UKprobab}
Under the Assumption \ref{A1}, for any fixed $0<a<b<\infty$ and for $\P_p$-a.s. environment $\omega$, we have
\begin{equation}
\lim_{n\to\infty}\sup_{|x|_1\le n} \sup_{\substack{m\in\N\\ an\le m\le bn}}|m\p^\omega_x(Z_m>0) - \frac{2}{p\sigma^2_\circ + (1-p)\sigma^2_\bullet}|=0.
\end{equation}

\end{thm}

Next result is the quenched Yaglom theorem.

\begin{thm}\label{thm: Yaglom}
Under the Assumption \ref{A1}, for any $x\in\Z^d$ and for $\P_p$-a.s. $\omega$,
\[
 \lim_{n\to\infty}\p_x^\omega\!\left(
 \frac{Z_n}{\{p\sigma_\circ^2+(1-p)\sigma_\bullet^2\}n/2}\le t
 \ \middle|\ Z_n>0\right)
 =F_{\mathrm{Exp}(1)}(t),\quad \forall t\in\R,
\]
where
\[
 F_{\mathrm{Exp}(1)}(t)=
 \begin{cases}
 0,&t<0,\\
 1-e^{-t},&t\ge0.
 \end{cases}
\]
\end{thm}

The usual approach to proving a Yaglom theorem is to compute the moments or the generating function of $Z_n$. However, this becomes difficult in a random environment where the offspring law depends on the spatial location. In this work, we shall prove the weak convergence to the exponential distribution using Stein's method. More precisely, our proof relies on Theorem 2.1 of \cite{PR} and Theorem 3.5 of \cite{CJP}, which provides an upper bound for the Wasserstein distance between the exponential law and the law of any non-negative random variable \(W\) with finite second moment.

The remainder of this paper is organized as follows. In Section \ref{prel}, we collect and prove several results for random walk with absorption and random walk in random scenery, we also introduce the spinal decomposition and Lyons' change of measure under quenched law. In Section \ref{survival}, we study the survival probabilities and prove Theorems \ref{thm: Sprobab} and \ref{thm: Sprobab+killing}. In Section \ref{critical}, we deal with the Kolmogorov estimate and Yaglom theorem, establishing Theorems \ref{thm: UKprobab} and \ref{thm: Yaglom}. Section \ref{lems} contains proofs of several technical lemmas. 

We use $c_0, c_1,c_2,\cdots$ to represent positive and finite real constants, and use $c_i(\varepsilon)$ for constants depending on $\varepsilon$. As usual, $a_n =o_n(1)$ means that $a_n\to 0$ as $n\to\infty$. For given environment $\omega$, $c_i^\omega$ and $c_i^\omega(\varepsilon)$ are used for constants depending on $\omega$. $a_n=o_n^\omega(1)$ means that $a_n$ converges $\P_p$-a.s. to zero as $n\to\infty$.

\section{Preliminaries}\label{prel}

\subsection{Preliminary results on random walk with absorption}

In this section, we study the simple random walk with killing or absorption when leaving some finite and connected set $A\subset \Z^d$. 

Take $x\in\Z^d$. Let $(S_n)_{n\ge0}$ be the simple random walk in $\Z^d$, with motion law $\nu$, starting from $S_0=x$. Denote its law by $\Q_x$. For the finite and connected set $A\subset\Z^d$, define the exiting time from $A$ by 
\begin{equation*}
\tau_A:=\inf\{ n\ge 0: S_n\notin A\}.
\end{equation*}
We are interested in $(S_n)$ killed at $\tau_A$. Then the corresponding transition operator $T_A$ is defined on $\{f: A \to \R\}$ such that
\[
(T_Af)(y) := \frac{1}{2d}\sum_{|e|_1=1} f(y+e) \ind{y+e\in A}, \quad \forall y\in A.
\]
In other words, for any $x,y\in A$, and for $f(\cdot) = \delta_x(\cdot)$,
\[
(T_A\delta_x)(y) = \Q_x(S_1=y) = \frac{1}{2d} \ind{|y-x|_1=1}.
\]
So, $T_A$ can be viewed as a symmetric matrix, whose all eigenvalues are real numbers. Let $\rho_A$ be its spectral radius, i.e.,
\[
\rho_A:= \max\{ |\lambda | : \lambda \textrm{ is an eigenvalue of } T_A\}.
\]
Since $T_A$ is sub-Markovian, we always have $0\le \rho_A\le 1$. In fact, we could obtain the following lemma. 

\begin{lem}\label{spectralradius}
If $A$ is finite and connected subset in $\Z^d$ with $\#A\ge 2$, then we have
\begin{enumerate}
\item $\rho_A\in (0,1)$.
\item for any $x\in A$, the limits 
\[
\lim_{n\to\infty} \frac{\Q_x(\tau_A>2n)}{\rho_A^{2n}} \textrm{ and } \lim_{n\to\infty} \frac{\Q_x(\tau_A> 2n +1)}{\rho_A^{2n+1}},
\]
exist and belong to $(0,\infty)$. In general, these two limits are different.
\item 
For every sequence of
finite connected sets $(A_k)$ with $R_{A_k}\to\infty$, we have
$\rho_{A_k}\to1$,
where $R_A$ denotes the radius of the largest $\ell^1$-ball contained in $A$, i.e., 
\[
R_A:= \max\{ R\in \N \vert \exists y\in A \textrm{ such that } \{z\in \Z^d: |z-y|_1\le R\}\subset A\}.
\]
\end{enumerate}

\end{lem}

We will prove this lemma in Section \ref{lems}. Based on this lemma, we establish a criterion for the survival probability of a branching random walk in $A$ with absorption when exiting $A$. 

Consider a branching random absorbed outside $A$ constructed as follows. Let $\mu = \sum_{k\in\N} p_k\delta_k$ be an offspring law. An ancestor particle is located at $x\in A$ at time $0$. At each time $n\ge1$, each alive particle is replaced independently by a random number of children according to $\mu$, and each newborn particle performs an independent simple random walk jump from its birth position. If a particle steps outside $A$, it gets killed immediately. The system continues to evolve like this.  Denote its law by $\p_x^{(A)}$. With a little abuse of notation, we still denote this process by $(S_u, u\in\T)$ and denote the number of alive particles at time $n$ by $Z_n$. Then the survival probability is 
\[
\p^{(A)}_x(S) = \p^{(A)}_x(Z_n \ge 1, \forall n\ge 0).
\]
Naturally, this model can be viewed as a multitype Galton-Watson tree by considering the position of a particle as its type. When $A$ is finite, this is a multitype system with finitely many types. We thus manage to establish the following proposition. Its proof is postponed to Section \ref{lems}.

\begin{prop}\label{brw+killing}
If $A$ is finite and connected subset in $\Z^d$ with $\#A\ge 2$, then for the branching random walk absorbed outside $A$, we have
\[
\p^{(A)}_x(Z_n \ge 1, \forall n\ge 0)>0, \forall x\in A,
\]
if and only if 
\begin{equation}
\left(\sum_k kp_k \right) \times \rho_A >1.
\end{equation}
\end{prop}

Note that, if the branching mechanism is supercritical, i.e., $\sum_k kp_k\in(1,\infty]$, then for all large enough $\ell^1-$ball $B$ in $\Z^d$, the branching random walk absorbed outside $B$ survives with positive probability. 

\subsection{Preliminary results on random walk in random scenery}\label{RWRS}

In this section, we recall some classical results on random walk in random scenery. In our setting, $\omega = (\omega(x))_{x\in\Z^d}$ are i.i.d. Bernoulli random variables with parameter $p\in(0,1)$. Recall that $(S_n)_{n\ge0}$ is simple random walk in $\Z^d$ under $\Q_x$. Let $\mathbb{Q}_x=\Q_x \otimes \P_p$ so that $\omega$ and $(S_n)_{n\ge0}$ are independent under $\mathbb{Q}_x$. The associated \textit{random walk in random scenery} is defined by 
\[
T_n := \sum_{j=0}^{n-1} \omega(S_j), \forall n\ge 1.
\]
It is well known that $\mathbb{Q}_x(T_n\in\cdot)=\mathbb{Q}_0(T_n\in\cdot)$ and that $\mathbb{Q}_x$-a.s., $\frac{T_n}{n}$ converges to $p$, see for instance \cite{KS}. Further, Gantert et al. \cite{GKS} established the annealed large deviation principle for $T_n$. We state the following theorem here.

\begin{thm}[ Theorem 1.3 of \cite{GKS}]\label{LDP-RWRS}
For any $\varepsilon\in (0, p\wedge (1-p))$, there exist positive constants $c_1(\varepsilon)$ and $c_2(\varepsilon)$ such that for any $x\in\Z^d$,
\begin{align*}
\lim_{n\to\infty} n^{-\frac{d}{d+2}}\log \mathbb{Q}_x\left( \frac{T_n}{n} > p+\varepsilon\right) = &- c_1(\varepsilon),\\
\lim_{n\to\infty} n^{-\frac{d}{d+2}}\log \mathbb{Q}_x\left( \frac{T_n}{n} < p-\varepsilon\right) = &- c_2(\varepsilon).
\end{align*}
\end{thm}

Note that for any $x\in\Z^d$, $\Q_x=\Q_0$. From this annealed large deviation theorem, we thus deduce the quenched strong law of large number for $T_n$ under the quenched law $\mathbb{Q}^\omega_x := \mathbb{Q}_x(\cdot\vert \omega)$.

\begin{lem}\label{Q-SLLN}
For any $x\in\Z^d$, for $\P_p$-a.s. $\omega$, we have
\[
\mathbb{Q}_x^\omega\left( \lim_{n\to\infty} \frac{T_n}{n} = p\right) =1.
\]
\end{lem}

\begin{proof}
For any  $\varepsilon\in (0, p\wedge (1-p))$ and any positive sequence $(\delta_n)_{n\ge1}$, by Markov inequality and \eqref{LDP-RWRS}, one sees that 
\begin{align*}
\P_p\left( \mathbb{Q}_x^\omega(|T_n - pn |\ge \varepsilon n) \ge \delta_n \right) \le & \frac{1}{\delta_n} \mathbb{Q}_x( |T_n - pn | \ge \varepsilon n) \\
\le & \frac{c_3(\varepsilon)}{\delta_n} e^{- c_0(\varepsilon) n^{\frac{d}{d+2}}},
\end{align*}
with $c_3(\varepsilon)>0$ and $c_0(\varepsilon) = \frac12[ c_1(\varepsilon)\wedge c_2(\varepsilon)]>0$. Taking $\delta_n = e^{- c_0(\varepsilon) n^{\frac{d}{d+2}}/2}$ yields that 
\[
\sum_{n\ge1} \P_p\left( \mathbb{Q}_x^\omega(|T_n - pn |\ge \varepsilon n) \ge \delta_n \right) <\infty.
\]
Borel-Cantelli Lemma shows that
\[
\P_p(\mathbb{Q}_x^\omega(|T_n - pn |\ge \varepsilon n) \ge \delta_n \textrm{ i. o. })=0.
\]
In other words, $\P_p$-a.s.,
\[
\mathbb{Q}_x^\omega(|T_n - pn |\ge \varepsilon n) \le c_0^\omega(\varepsilon) \delta_n, \forall n\ge1.
\]
Again, as $\mathbb{Q}_x^\omega(|T_n - pn |\ge \varepsilon n) <\infty$, $\P_p$-a.s., we conclude Lemma \ref{Q-SLLN}.
\end{proof}
From the proof above, one sees that for $\varepsilon\in(0,1)$ sufficiently small, there exists an environment r.v. $c_0^\omega(\varepsilon)\in(0,\infty)$ such that $\P_p$-a.s. 
\begin{equation}\label{RWRS-LDP}
\max_{|x|_1\le n^2}\mathbb{Q}_x^\omega\left(| T_n -pn | > \varepsilon  n \right) \le c_0^\omega(\varepsilon) e^{- \frac14 c_0(\varepsilon) n^{\frac{d}{d+2}}}, \forall n\ge 1.
\end{equation}
\subsection{Lyons' change of measure and spinal decomposition}

In the section, we always impose the Assumption \ref{A1}. Then the system is in critical regime. Recall that $Z_n$ is the number of alive particles at time $n$. Then $\e^\omega(Z_n)=1$. Let $\mathcal{F}_n$ denote the natural filtration of the branching random walk, defined by
\[
\mathcal{F}_n:=\sigma( (u,S_u), | u | \le n),
\]
with $|u|$ denoting the generation of a particle $u$.

It is immediate that $\{Z_n\}_{n\ge0}$ is a martingale with respect to $\{\mathcal{F}_n\}_{n\ge0}$ under the quenched law $\p^\omega_x$, for any $x\in\Z^d$. So, given the environment $\omega$, for any $x\in\Z^d$, we introduce the following probability measure:
\begin{equation}
\frac{d\overline{\p}_x^\omega}{d \p^\omega_x}\vert_{\mathcal{F}_n} = Z_n, \forall n\in\N.
\end{equation}

We reconstruct the branching random walk under $\overline{\mathbb{P}}^\omega_x$ as follows. Given the environment $\omega$, we start with an ancestor particle at position $x$, which we designate as the distinguished particle $\zeta_0$ at generation $0$ and denote its position by $S_{\zeta_0} = x$. For each generation $n \ge 1$, the evolution proceeds as:
\begin{itemize}
    \item The distinguished particle $\zeta_{n-1}$, located at either an open or a closed site (according to $\omega$), gives birth to a random number $\xi_n$ of children. The law of $\xi_n$ is the size-biased distribution $\hat{\mu}_\circ = \sum_k k \mu_\circ(k) \delta_k$ if the site is open, and $\hat{\mu}_\bullet = \sum_k k \mu_\bullet(k) \delta_k$ if it is closed.
    \item Independently, every other particle of the $(n-1)$-th generation (if any) produces offspring according to the original distribution $\mu_\circ$ or $\mu_\bullet$, depending on whether its own site is open or closed.
    \item Among the children of the distinguished particle, one is chosen uniformly at random to become the new distinguished particle $\zeta_n$. All remaining children, together with the offspring of all non-distinguished particles, form the ordinary particles of the $n$-th generation.
    \item Immediately after birth, every newborn particle independently moves from its parent's location according to the jump distribution $\nu$.
\end{itemize}
By convention, the newborn particles are labeled in lexicographic order at each branching event. This procedure defines the quenched law, denoted by $\widehat{\p}^\omega_x$, of the branching random walk together with a distinguished line of descent $\{(u,S_u)_{u\in\T}; (\zeta_n, S_{\zeta_n})_{n\ge0}\}$. We call the distinguished line of descent $(\zeta_n)_{n\ge0}$ the spine. 

It is known that the marginal law of $\{(u,S_u)_{u\in\T}\}$ under $\widehat{\p}^\omega_x$ is exactly $\overline{\p}^\omega_x$. We also have the following proposition, whose proof can be referred to \cite{Shi}.

\begin{prop}\label{prop: spine}
For any $n\in\N$ and $x\in\Z^d$, the following assertions hold.
\begin{enumerate}
\item[\textbf{(1)}] 
\begin{equation}\label{prop-spine-1}
\widehat{\p}^\omega_x \vert_{\mathcal{F}_n} = \overline{\p}^\omega_x \vert_{\mathcal{F}_n} = Z_n \cdot \p^\omega_x \vert_{\mathcal{F}_n}.
\end{equation}
\item[\textbf{(2)}] For any $u\in\T$,
\[
\widehat{\p}^\omega_x( \zeta_n = u \vert \mathcal{F}_n ) = \frac{\ind{|u|=n}}{Z_n}.
\]
\item[\textbf{(3)}] Under $\widehat{\p}^\omega_x$, $(S_{\zeta_n})_{n\ge0}$ has the same distribution as $(S_{n})_{n\ge0}$ under $\Q_x$.
\end{enumerate}

\end{prop}

Under $\widehat{\p}^\omega_x$, we decompose the genealogical tree $\T$ relative to the spine $(\zeta_n)_{n\ge0}$. At each generation $j\ge1$, there are $\xi_j-1$ siblings of the spine particle $\zeta_j$, which, according to the lexicographic order, can be naturally divided into those lying to the left and to the right of the spine. Let $\mathrm{L}_j$ (resp. $\mathrm{R}_j$) be the set of the siblings to the left (resp. right) of $\zeta_j$. Set $l_j = \#\mathrm{L}_j$ and $r_j=\#\mathrm{R}_j$. For $n\ge j$, let $L_{n,j}$ (resp. $R_{n,j}$) be the number of descendants at generation $n$ of the left (resp. right) siblings of $\zeta_j$. Set $L_n:=  \sum_{j=1}^n L_{n,j}$ and $R_n:= 1 + \sum_{j=1}^n R_{n,j}$. We hence observe that
\begin{equation}\label{spine-decomp}
\xi_n= l_n +1+r_n, \textrm{ and } Z_n = \sum_{j=1}^n L_{n,j} + 1 + \sum_{j=1}^n R_{n,j}= L_n +R_n.
\end{equation}
 Note that $L_n$ counts the number of particles at $n$-th generation which are on the left of $\zeta_n$ in the lexicographic order. 

From the above construction, it follows immediately that for any $j\ge 1$, $k,m\in \N$,
\begin{equation}\label{lawlr}
\begin{cases} 
\widehat{\p}^\omega_x \left( l_j = k, r_j = m \vert \omega(S_{\zeta_{j-1}} )=1 \right) =& \mu_\circ(k+m+1) ,\\
\widehat{\p}^\omega_x \left( l_j = k, r_j = m \vert \omega(S_{\zeta_{j-1}} )=0 \right) =& \mu_\bullet(k+m+1) .
\end{cases}
\end{equation}
Consequently,
\begin{equation}\label{meanlr}
\begin{cases}
\widehat{\e}^\omega_x[ l_j \vert  \omega(S_{\zeta_{j-1}} )=1 ] =\widehat{\e}^\omega_x[ r_j \vert  \omega(S_{\zeta_{j-1}} )=1 ] =& \frac12\sigma^2_\circ, \\
\widehat{\e}^\omega_x[ l_j \vert  \omega(S_{\zeta_{j-1}} )=0 ] = \widehat{\e}^\omega_x[ r_j \vert  \omega(S_{\zeta_{j-1}} )=0 ] =& \frac12\sigma^2_\bullet.
 \end{cases}
\end{equation}
We write $\widehat{\p}^\omega_{(x_j)_{j\le n}}(\cdot)$ for $\widehat{\p}^\omega[\cdot\vert (S_{\zeta_j}) =(x_j)_{j\le n} ]$ where $(x_0,x_1,\cdots, x_n)$ is an admissible nearest-neighbor path in $\Z^d$ and use $\widehat{\e}^\omega_{(x_j)_{j\le n}}$ for the corresponding expectation. We also have the following lemma on the law of $(l_j,r_j)$.
\begin{lem}
There exist constants $C_1,C_2\in(0,\infty)$ such that for any $\delta\in(0,1]$ and any admissible path $(x_j)_{j\le n}\in(\Z^d)^n$, one has for any $1\le j\le n$,
\begin{equation}\label{bdlr}
C_1\widehat{\e}^\omega_{(x_j)_{j\le n}}[ (1-\delta)^{l_j}] \widehat{\e}^\omega_{(x_j)_{j\le n}}[ r_j ] \le \widehat{\e}^\omega_{(x_j)_{j\le n}}[ (1-\delta)^{l_j} r_j ] \le C_2 \widehat{\e}^\omega_{(x_j)_{j\le n}}[ (1-\delta)^{l_j}  ] \widehat{\e}^\omega_{(x_j)_{j\le n}}[ r_j  ] .
\end{equation}
Further, as $\delta\downarrow0+$, uniformly in $j,n$ and $(x_j)_{j\le n}\in(\Z^d)^n$,
\begin{equation}\label{asymplr}
\widehat{\e}^\omega_{(x_j)_{j\le n}} [(1-\delta)^{l_j} r_j ] =  \widehat{\e}_{(x_j)_{j\le n}}^\omega \left[(1-\delta)^{l_j} \right] \widehat{\e}_{(x_j)_{j\le n}}^\omega \left[ r_j \right] + o_\delta(1).
\end{equation}
\end{lem}

\begin{proof}
In fact, if $\mu_\bullet = \delta_1$ and $\omega(x_{j-1})=0$, $(l_j,r_j)$ is trivially zero vector by \eqref{lawlr}. So, \eqref{bdlr} always hold. Otherwise, $(l_j,r_j)$ is not trivially zero vector and we deduce from \eqref{lawlr} that
\begin{align*}
\widehat{\e}_{(x_j)_{j\le n}}^\omega \left[(1-\delta)^{l_j} r_j \right] \ge & \widehat{\e}_{(x_j)_{j\le n}}^\omega \left[\ind{l_j=0} r_j \right] = 
\begin{cases}
\sum_{m\ge0} m \mu_\circ(m+1) = \mu_\circ(0), & \textrm{ if } \omega(x_{j-1} ) =1;\\
\sum_{m\ge0} m \mu_\bullet(m+1) = \mu_\bullet(0), & \textrm{ if } \omega(x_{j-1} ) =0.
\end{cases}
\end{align*}
On the other hand,
\begin{align*}
\widehat{\e}_{(x_j)_{j\le n}}^\omega \left[(1-\delta)^{l_j} \right] \widehat{\e}_{(x_j)_{j\le n}}^\omega \left[ r_j \right] \le \widehat{\e}_{(x_j)_{j\le n}}^\omega \left[ r_j \right] =
\begin{cases}
\frac12\sigma^2_\circ, & \textrm{ if } \omega(x_{j-1})=1;\\
\frac12\sigma^2_\bullet, & \textrm{ if } \omega(x_{j-1})=0.
\end{cases}
\end{align*}
So, we can always find a constant $C_1>0$ such that $\mu_\circ(0)\ge \frac{C_1}{2}\sigma^2_\circ$ and $\mu_\bullet(0)\ge \frac{C_1}{2}\sigma^2_\bullet$. This leads to the lower bound in \eqref{bdlr}. For the upper bound, note that
\begin{align*}
\widehat{\e}_{(x_j)_{j\le n}}^\omega \left[(1-\delta)^{l_j} r_j \right] \le & \widehat{\e}_{(x_j)_{j\le n}}^\omega \left[ r_j \right] ,
\end{align*}
and that
\begin{align*}
\widehat{\e}_{(x_j)_{j\le n}}^\omega \left[(1-\delta)^{l_j} \right] \ge \widehat{\e}_{(x_j)_{j\le n}}^\omega \left[ \ind{l_j=0} \right] =
\begin{cases}
\sum_{m\ge 0}\mu_\circ(m+1)= 1-\mu_\circ(0), & \textrm{ if } \omega(x_{j-1})=1;\\
\sum_{m\ge0}\mu_\bullet(m+1) = 1- \mu_\bullet(0), & \textrm{ if } \omega(x_{j-1})=0.
\end{cases} 
\end{align*}
Therefore, we can take $C_2= \frac{1}{1-\mu_\circ(0)\vee \mu_\bullet(0)}>0$ to ensure the upper bound in \eqref{bdlr}. 

On the other hand, dominated convergence theorem shows that
\begin{align*}
\widehat{\e}_{(x_j)_{j\le n}}^\omega \left[(1-\delta)^{l_j} r_j \right] = & \widehat{\e}_{(x_j)_{j\le n}}^\omega \left[ r_j \right] + o_\delta(1),\\
 \widehat{\e}_{(x_j)_{j\le n}}^\omega \left[(1-\delta)^{l_j} \right] \widehat{\e}_{(x_j)_{j\le n}}^\omega \left[ r_j \right] = & \widehat{\e}_{(x_j)_{j\le n}}^\omega \left[ r_j \right] + o_\delta(1).
\end{align*}
This implies then \eqref{asymplr} as $o_\delta(1)$ only depends on the two laws $\mu_\circ$ and $\mu_\bullet$.
\end{proof}
In addition, for any particle $u\in\T$ with $|u|\le n$, let $Z_n^{(u)}$ be the number of its descendants at $n$-th generation. Observe that 
\begin{equation}\label{subtree}
L_{n,j}=\sum_{u\in\mathrm{L}_j} Z^{(u)}_n, \quad R_{n,j} = \sum_{u\in\mathrm{R}_j} Z^{(u)}_n,
\end{equation}
where conditioned on $\mathcal{G}_\infty:=\sigma\{(\zeta_0, S_{\zeta_0}), (\zeta_j, S_{\zeta_j}, (u, S_u)_{u\in \mathrm{L}_j\cup\mathrm{R}_j})_{j\ge1}\}$, for $u\in \mathrm{L}_j\cup\mathrm{R}_j$, $Z_n^{(u)}$ are independent under $\widehat{\p}^\omega_x$, and distributed as $Z_{n-j}$ under $\p^\omega_{S_u}$. 

Next, let us take a look at the spine. Once the trajectory $(S_{\zeta_j})_{j\ge0}$ is given, along the spine, at the generations belonging to $I_\circ(n):=\{j\in\N\cap[1,n]: \omega(S_{\zeta_{j-1}}) =1 \}$, branching happens with $\xi_j$ distributed as $\widehat{\mu}_\circ$, whereas at the generations belonging to $I_\bullet(n):=\{j\in\N\cap[1,n]: \omega(S_{\zeta_{j-1}}) =0 \}$, branching happens with $\xi_j$ distributed as $\widehat{\mu}_\bullet$. This leads to 
\begin{equation}
\frac1n \sum_{1\le j\le n} r_j = \frac{\sum_{j\in I_\circ(n)} r_j }{\# I_\circ(n)}\ind{I_\circ(n)\neq \emptyset} \times \frac{\sum_{j=1}^n \omega(S_{\zeta_{j-1}})}{n} + \frac{\sum_{j\in I_\bullet(n)} r_j }{\#I_\bullet(n)}\ind{I_\bullet(n) \neq \emptyset} \times \frac{n - \sum_{j=1}^n \omega(S_{\zeta_{j-1}})}{n},
\end{equation}
with the convention that $\frac{\sum_{j\in I_\circ(n)} r_j }{\# I_\circ(n)}\ind{I_\circ(n)\neq \emptyset} =0$ if $I_\circ(n)=\emptyset$.
In view of the quenched strong law of large number in Lemma \ref{Q-SLLN}, for $\P_p$-a.s, $\omega$, 
\begin{equation}\label{asym-spine}
\frac{\#I_\circ(n)}{n} =  \frac{\sum_{j=1}^n \omega(S_{\zeta_{j-1}})}{n} \xrightarrow[n\to\infty]{ \widehat{\p}^\omega_x -a.s. } p,\quad \frac{\#I_\bullet(n)}{n} =  \frac{n-\sum_{j=1}^n \omega(S_{\zeta_{j-1}})}{n} \xrightarrow[n\to\infty]{ \widehat{\p}^\omega_x -a.s. } 1-p.
\end{equation}
By \eqref{lawlr} and \eqref{meanlr}, for $j\in I_\circ(n)$, $r_j$ are i.i.d. with mean $\frac12\sigma^2_\circ$ whereas for $j\in I_\bullet(n)$, $r_j$ are i.i.d. with mean $\frac12\sigma^2_\bullet$. The classical strong law of large number shows that for $\P_p$-a.s. $\omega$,  $ \frac{\sum_{j\in I_\circ(n)} r_j }{\# I_\circ(n)} $ converges $\widehat{\p}^\omega_x$-a.s. to $\frac12\sigma^2_\circ$ and $ \frac{\sum_{j\in I_\bullet(n)} r_j }{\# I_\bullet(n)} $ converges $\widehat{\p}^\omega_x$-a.s. to $\frac12\sigma^2_\bullet$.

As a consequence, we obtain the following convergence result.

\begin{lem}\label{LLNlr}
For any $x\in\Z^d$ and for $\P_p$-a.s. environment $\omega$, 
\begin{equation*}
\frac1n \sum_{1\le j\le n} r_j \xrightarrow[n\to\infty]{ \widehat{\p}^\omega_x-a.s. } \frac{p\sigma^2_\circ + (1-p)\sigma^2_\bullet}{2}.
\end{equation*}
The similar result holds for $(l_j)_{j\ge1}$. 
\end{lem}
%Apparently, the same result holds for $(l_j)_{j\ge1}$. 

%Let us also give some large deviation estimate for $(r_j)_{j\ge1}$. Take $\delta\in(0,1)$ sufficiently small, we show that
%\begin{equation}
%\p^\omega_x( \sum_{j=1}^n r_j \le \delta n ) \le C_\omega e^{- C_\delta n^{\frac{d}{d+2}}}
%\end{equation}
%The proof of Lemma \ref{LLWlr} will be given in Section \ref{lems}.

At the end of this section, we introduce a useful identity. For any $n\ge1$, on $\{Z_n \ge 1\}$, let $u_n^{\textrm{leftmost}}$ denote the leftmost particle at the $n$-th generation with respect to the lexicographic order. Observe that
\[
\ind{Z_n \ge 1} = \sum_{|u|=n} \ind{u = u_n^{\textrm{leftmost}}}.
\]
Take arbitrary $A\in\mathcal{F}_n$, by Proposition \ref{prop: spine}-(1),
\begin{align*}
\e^\omega_x\left[ \ind{Z_n\ge 1} 1_A \right] =& \e^\omega_x\left[ \sum_{|u|=n} \ind{u = u_n^{\textrm{leftmost}}} 1_A \right]\\
= & \widehat{\e}^\omega_x\left[ \sum_{|u|=n} \frac{1}{Z_n} \ind{u = u_n^{\textrm{leftmost}}} 1_A \right]
\end{align*}
which by Proposition \ref{prop: spine}-(2) implies that
\begin{align*}
\e^\omega_x\left[ \ind{Z_n\ge 1} 1_A \right] =& \widehat{\e}^\omega_x\left[ \sum_{|u|=n} \widehat{\p}^\omega_x(\zeta_n = u \vert \mathcal{F}_n) \ind{u = u_n^{\textrm{leftmost}}} 1_A \right] = \widehat{\e}^\omega_x\left[  \ind{\zeta_n = u_n^{\textrm{leftmost}}} 1_A \right].
\end{align*}
Note that under $\widehat{\p}^\omega_x$, $\ind{\zeta_n = u_n^{\textrm{leftmost}}} = \ind{L_n=0}$. Therefore, for all $A\in\mathcal{F}_n$, 
\begin{equation}\label{keyeq}
\e^\omega_x\left[ \ind{Z_n\ge 1} 1_A \right] =  \widehat{\e}^\omega_x\left[  \ind{L_n=0} 1_A \right].
\end{equation}

\section{Survival probabilities: proof of Theorem \ref{thm: Sprobab} and Theorem \ref{thm: Sprobab+killing}}\label{survival}

We study the quenched survival probability of the branching random walk in IID environment in this section. By excluding the degenerate case with $\mu_\circ = \mu_\bullet=\delta_1$, there exist three regimes:
\begin{itemize}
\item \textbf{(Sub)-critical branchings}
\begin{equation}\label{sub}
\max\{ m_\circ, m_\bullet \} \le 1, \textrm{ and } \min\{\mu_\circ(1), \mu_\bullet(1)\}<1.
\end{equation}
\item \textbf{Supercritical branching with soft killing}
\begin{equation}\label{sup+softk}
\max\{m_\circ, m_\bullet \} >1, \textrm{ and } \max\{ \mu_\circ(0), \mu_\bullet(0) \} <1.
\end{equation}
\item \textbf{Supercritical branching with hard killing}
\begin{equation}\label{sup+hardk}
\max\{ m_\circ, m_\bullet \} >1, \textrm{ and } \max\{\mu_\circ(0), \mu_\bullet(0) \} =1.
\end{equation}
\end{itemize}

We will deal with these three regimes separately in three subsections.

\subsection{Proof of Theorem \ref{thm: Sprobab}-(1)}

Now we work under \eqref{sub} and prove that extinction happens a.s. in this regime. First, observe that as $m_\circ\vee m_\bullet \le 1$, for any environment $\omega$,
\begin{equation}
\e^\omega_x[ Z_{n+1} \vert \mathcal{F}_n ] =\sum_{|u|=n}[ \ind{\omega(S_u)=1}m_\circ + \ind{\omega(S_u)=0}m_\bullet ] \le Z_n, \forall n\ge0.
\end{equation}
This means that $(Z_n)_{n\ge0}$ is a non-negative supermartingale with respect to $\{\mathcal{F}_n\}_{n\ge0}$. Therefore, $Z_n$ converges a.s. to some non-negative limit $Z_\infty\in[0,\infty)$. Fatou's Lemma shows that
\[
\e^\omega_x[Z_\infty] \le \liminf_{n\to\infty} \e^\omega_x[Z_n] \le 1.
\]
Moreover, as all $Z_n$ are $\N$-valued r.v.'s, so it the a.s. limit $Z_\infty$. 

We need to exclude the situation that $Z_\infty \ge1$. 
Fix $k\ge1$ and let
\[
 E_{m,k}:=\{Z_n=k\text{ for every }n\ge m\}.
\]
First suppose that both offspring laws differ from $\delta_1$.  Conditional on
$\mathcal F_n$ and $Z_n=k$, let $r$ be the number of particles at open sites.
Then $Z_{n+1}$ is the sum of $r$ independent variables with law $\mu_\circ$
and $k-r$ independent variables with law $\mu_\bullet$.  For each
$r\in\{0,\ldots,k\}$ this sum is non-degenerate, and hence
\[
 q_k:=\max_{0\le r\le k}\underbrace{\mu_\circ\ast\mu_\circ\ast\cdots\mu_\circ}_{r\textrm{-times}}\ast\underbrace{\mu_\bullet\ast\cdots\ast\mu_\bullet}_{(k-r)\textrm{-times}}(k)<1.
\]
Consequently, conditional iteration gives
\[
 \p_x^\omega(Z_{m+1}=\cdots=Z_{m+N}=k\mid\mathcal F_m)
 \le q_k^N\quad\text{on }\{Z_m=k\},
\]
and therefore $\p_x^\omega(E_{m,k})=0$.

It remains to consider the case in which one law is $\delta_1$.  Call its
sites \emph{passive} and call the other sites \emph{active}; let $\mu$ be the
non-degenerate active law.  If, while $Z_n=k$, exactly $r\ge1$ particles are at
active sites, then
\[
 \p_x^\omega(Z_{n+1}=k\mid\mathcal F_n)
 =\underbrace{\mu\ast\cdots\ast\mu}_{r\textrm{-times}}(r)
 \le q_k':=\max_{1\le r\le k}
 \underbrace{\mu\ast\cdots\ast\mu}_{r\textrm{-times}}(r)<1.
\]
Thus $E_{m,k}$ has probability zero if active sites are occupied at infinitely
many generations.  On the complementary possibility, from some generation
onward all $k$ particles lie at passive sites.  Passive reproduction produces
exactly one child, so, until an active site is reached, the $k$ descendants
follow independent simple random walks.  Lemma~\ref{Q-SLLN}, applied on the
countable intersection over all starting sites, implies that for
$\P_p$-a.e. environment each such walk visits the active set with positive
limiting frequency (equal to $p$ or $1-p$).  Hence the probability that all
of them avoid the active set forever is zero.  This excludes $E_{m,k}$ in the
remaining case as well.

Finally, take the countable union over $m\ge1$ and $k\ge1$.  It follows that
$\p_x^\omega(Z_\infty\ge1)=0$, and hence $Z_\infty=0$ almost surely. 

\subsection{Proof of Theorem \ref{thm: Sprobab}-(2)}

In this part, we impose the assumption \ref{sup+softk}. 

Without loss of generality, we suppose that $m_\circ>1$. Consider the branching random walk with offspring law $\mu_\circ$, starting from $x\in\Z^d$ and killed upon leaving the $\ell^1$-ball $B_x(R):=\{z\in\Z^d: | z - x |_1 \le R \}$ with some fixed $R\ge 1$. Then it follows from Proposition \ref{brw+killing} that for $R\ge R_{m_\circ}$ sufficiently large, this system survives with positive probability. Take $R= R_{m_\circ}$ and denote the corresponding survival probability by $\mathrm{p}_{\mu_\circ}\in(0,1]$.

Going back to the proof of Theorem \ref{thm: Sprobab}-(2), we first claim that $\P_p$-a.s.,
\[
\exists x\in \Z^d \textrm{ such that } \omega\vert_{B_x(R)} \equiv 1.
\]
In fact, $\Z^d$ can be decomposed into a countable family of pairwisely disjoint hypercubes of side length $2R$. It is obvious that, with probability one, at least one such hypercube is entirely open. As every hypercube contains a $\ell^1$-ball $B_x(R)$, the claim is proved. 

Since $\mu_\circ(0) \vee \mu_\bullet(0)<1$, the killings are soft. Given the environment $\omega$ for which we have an open ball $B_{x(\omega)}(R)$, run a branching random walk starting from the origin and consider the following event $E$ associated with a finite trajectory from $0$ to $x(\omega)$: $\{x_0=0, x_1,\cdots, x_K= x(\omega)\}$:
\begin{itemize}
\item the ancestor $\ell_0$ has at least one child, and the leftmost one is denoted by $\ell_1$ which moves from $x_0$ to $x_1$;
\item at time $1\le n< K= K(\omega)$, $\ell_n$ has at least one child, and the leftmost one is denoted by $\ell_{n+1}$ which moves from $x_n$ to $x_{n+1}$. 
\end{itemize}
It is clear that once a particle arrives at $x(\omega)$, it could produce a branching random walk in $B_{x(\omega)}(R)$ which survives with positive probability. We then end up with
\[
\p^\omega_0(\S) \ge \p^\omega_0(E) \p^{(B_{x(\omega)}(R))}_{x(\omega)}(\S)\ge (\frac{1}{2d})^{K(\omega)} [(1-\mu_\circ(0))\wedge (1-\mu_\bullet(0))]^{K(\omega)}\mathrm{p}_{\mu_\circ} >0.
\]

\subsection{Proof of Theorem \ref{thm: Sprobab}-(3) and Theorem \ref{thm: Sprobab+killing}}

In this part, we deal with the case \ref{sup+hardk} where we have supercritical branching and hard killing. Without loss of generality, assume that $m_\circ>1$ and $\mu_\bullet(0) =1 $. 

Recall that for Bernoulli site percolation $\omega = (\omega(x), x\in \Z^d)$, the open cluster of the origin is $\mathcal{C}(0)$ so that for any $x\in \mathcal{C}(0)$, we can find a simple random walk trajectory consisting of open sites, from $0$ to $x$. For the branching random walk started from the origin, due to the hard killing at closed sites, one actually gets a branching random walk with offspring law $\mu_\circ$, getting killed upon leaving $\mathcal{C}(0)$. When $\#\mathcal{C}(0)<\infty$, $\mathcal{C}(0)$ is a finite and connected subset of $\Z^d$, the positivity of survival depends on whether $\rho_{\mathcal{C}(0)} m_\circ$ exceeds $1$, according to Proposition~\ref{brw+killing}.

In particular, if $\mathcal{C}(0) = \{0\}$, extinction happens always. It then follows that
\begin{align*}
\P_p(\omega: \p^\omega_0(\S)=0) \ge \P_p( \mathcal{C}(0) = \{0\} ) = p(1-p)^{2d}>0.
\end{align*}
On the other hand, it is known from Lemma \ref{spectralradius} that there exists $R_{m_\circ}\ge 1$ such that for any $R\ge R_{m_\circ}$, $\rho_{B_0(R)} m_\circ >1$. This shows that if $B_0(R)\subset \mathcal{C}(0)$ with $R$ sufficiently large, the system survives with probability larger than $\P_0^{B_0(R)}(\S)>0$. Consequently, 
\begin{align*}
\p_0(\S) \ge & \p_0\left( \S\cap \{B_0(R) \subset C(0)\} \right) = \E_p\left[ \ind{B_0(R) \subset C(0)} \p^\omega_0(\S) \right]\\
 \ge & \P_0^{B_0(R)}(\S) \P_p( B_0(R) \subset C(0) ) >0.
\end{align*}
We thus conclude Theorem  \ref{thm: Sprobab}-(3) . 

Let us prove Theorem \ref{thm: Sprobab+killing}. For the finite cluster case, by Proposition \ref{brw+killing}, we immediately have
\[
\{\omega: \p^\omega_0(\S)>0\}\cap\{\#\mathcal{C}(0) <\infty \} \overset{ \P_p-a.s. }{=} \{\omega: \#\mathcal{C}(0)<\infty, \rho_{\mathcal{C}(0)} m_\circ > 1\}.
\]
For the infinite cluster case, we will check the following Lemma which says that the infinite open cluster contains balls of arbitrarily large radius. 

\begin{lem}\label{percolation}
If $\{\#\mathcal{C}(0) = \infty\}$ happens with positive probability, i.e., $\theta_{\Z^d}(p)>0$, then
\begin{equation}
\P_p\left( \forall R\ge1, \exists x\in \mathcal{C}(0), \textrm{ s.t. } B_x(R) \subset \mathcal{C}(0) \vert \#\mathcal{C}(0) = \infty \right) =1.
\end{equation}
\end{lem}

The similar result holds for Bernoulli bond percolation in $\Z^d$, see for instance Lemma 4.3.1 in \cite{Trees-Networks}. The proof of Lemma \ref{percolation} is postponed to Section \ref{lems}. As a consequence of Lemma \ref{percolation} and Proposition \ref{brw+killing}, one sees that on $\{\#\mathcal{C}(0) = \infty\}$, a.s.,
\[
\p^\omega_0(\S) >0. 
\]

Combining the finite cluster case and the infinite cluster case, we obtain that
\begin{align*}
\P_p(\omega: \p^\omega_0(\S) > 0) = & \P_p(\omega: \#\mathcal{C}(0) = \infty ) + \P_p(\{\omega: \p^\omega_0(\S) > 0\}\cap\{\#\mathcal{C}(0) < \infty \})  \\
= & \theta_{\Z^d}(p) + \P_p\left( \{\omega: \#\mathcal{C}(0)<\infty, \rho_{\mathcal{C}(0)} m_\circ > 1\}\right),
\end{align*}
which completes the proof of Theorem \ref{thm: Sprobab+killing}. 

\section{Kolmogorov estimate and Yaglom theorem in critical case}\label{critical}

This section is devoted to dealing with the critical regime under Assumption \ref{A1}. We will establish the Kolmogorov estimate, i.e., Theorem \ref{thm: UKprobab} which generalizes the result of \cite{EP} for a special case. Moreover, we will prove the corresponding Yaglom Theorem \ref{thm: Yaglom} which confirms the conjecture in \cite{EP} in a more general setting.

Without loss of generality, we assume that $\mu_\circ\neq \delta_1$.

\subsection{Uniform Kolmogorov estimate: proof of Theorem \ref{thm: UKprobab}}

In this part, we study the decay of the quenched probability $\p^\omega_x(Z_n \ge 1) $ where $Z_n$ counts the number of alive particles at the $n$-th generation. By \eqref{keyeq}, one sees that
\begin{equation}\label{probZn}
p^\omega_n(x):= \p^\omega_x(Z_n \ge 1) = \widehat{\p}^\omega_x( L_n =0 ).
\end{equation}
By Theorem~\ref{thm: Sprobab}-(1), apparently, $p^\omega_n(x)\to 0$ as $n\to\infty$. We first prove that $p^\omega_n(x)$ is of order $\frac1n$, and then prove Theorem \ref{thm: UKprobab}.

\paragraph{Upper bound and lower bound for $p^\omega_x(n)$.} On the one hand, we can check recursively that $\e^\omega_x[Z_n^2] \le C_0 + \e^\omega_x[Z_{n-1}^2]$ for any $n\ge 1$, with $C_0:=\sigma^2_\circ\vee \sigma^2_\bullet$. This brings out that 
\begin{equation}\label{2mombd}
\sup_x\e^\omega_x[ Z_n ^2 ] \le C_0 n+ 1.
\end{equation}
Paley-Zygmund inequality shows that for any $x\in\Z^d$ and for all environment $\omega$,
\begin{equation}\label{lwbdp}
p^\omega_n(x) \ge \frac{ \e^\omega_x[Z_n]^2 }{ \e^\omega_x[Z_n^2] } \ge \frac{1}{C_0 n+1}, \forall n\ge 0.
\end{equation}
On the other hand, $\e^\omega_x[Z_n]=1$ for all $n\in\N$. Note that $Z_n\in\mathcal{F}_n$, again by \eqref{keyeq}, we have
\begin{align*}
1=\e^\omega_x[Z_n] = \e^\omega_x[ \ind{Z_n \ge 1} Z_n] = \widehat{\e}_x^\omega[ \ind{L_n =0} Z_n ].
\end{align*}
By use of \eqref{spine-decomp} and \eqref{subtree}, one gets that
\begin{align*}
1 = & \widehat{\e}_x^\omega[ \ind{L_n =0} (1+\sum_{j=1}^n R_{n,j}) ] = \widehat{\e}^\omega_x\left[   \widehat{\e}^\omega_x\left[  \ind{L_n =0} (1+\sum_{j=1}^n R_{n,j}) \vert \mathcal{G}_\infty \right] \right] \\
= & \widehat{\e}_x^\omega \left[ \prod_{j=1}^n \prod_{u\in \mathrm{L}_j}(1-p^\omega_{n-j}(S_u)) \times (1+ \sum_{j=1}^n r_j ) \right].
\end{align*}
Recall that $(S_{\zeta_j})_{j\ge0}$ is a simple random walk. For $u\in\mathrm{L}_j$, it makes an independent simple random walk jump from $S_{\zeta_{j-1}}$.  So, 
\begin{align}\label{mean-spineS}
&1=   \widehat{\e}_x^\omega \left[ \prod_{j=1}^n (1- p^\omega_{n-j}\ast \nu (S_{\zeta_{j-1}} ) )^{l_j} \times (1+ \sum_{j=1}^n r_j ) \right] \\
&= \sum_{(x_j)_{j\le n} \in x+\mathscr{S}_n} \widehat{\p}^\omega_x( (S_{\zeta_j})_{j\le n} = (x_j)_{j\le n}) \widehat{\e}_{(x_j)_{j\le n}}^\omega \left[ \prod_{j=1}^n (1- p^\omega_{n-j}\ast \nu (x_{j-1} ) )^{l_j} \times (1+ \sum_{j=1}^n r_j )  \right],\nonumber
%&\times  \widehat{\e}_x^\omega \left[ \prod_{j=1}^n (1- p^\omega_{n-j}\ast \nu (x_{j-1} ) )^{l_j} \times (1+ \sum_{j=1}^n r_j ) \Big \vert (S_{\zeta_j})_{1\le j \le n} = (x_j)_{1\le j\le n} \right] 
\end{align}
where $\mathscr{S}_m:=\{(x_j)_{j\le m} \in (\Z^d)^{m+1}: x_0=0, |x_j-x_{j-1}|_1=1, \forall 1\le j \le m\}$ is the collection of simple random walk path of length $m$ and
\begin{align*}
p^\omega_k\ast\nu( z ) = \sum_y p^\omega_k( z - y) \nu(y),\textrm{ and }  \widehat{\e}_{(x_j)_{j\le n}}^\omega \left[ \cdot \right] = \widehat{\e}^\omega \left[ \cdot \vert (S_{\zeta_j})_{0\le j \le n} = (x_j)_{0\le j\le n} \right] 
\end{align*}
In the same spirit, we also have
\begin{equation}\label{probab-spine}
p^\omega_n(x) =\widehat{\p}^\omega_x( L_n =0 )= \widehat{\e}_x^\omega \left[ \prod_{j=1}^n (1- p^\omega_{n-j}\ast \nu (S_{\zeta_{j-1}} ) )^{l_j} \right].
\end{equation}
Note that given the path $(S_{\zeta_j})_{0\le j\le n-1}$,  $(l_j,r_j)_{1\le j\le n}$ are independent random vectors. Thus,
\begin{align}\label{mean-spinelr}
&\widehat{\e}_{(x_j)_{j\le n}}^\omega \left[ \prod_{j=1}^n (1- p^\omega_{n-j}\ast \nu (x_{j-1} ) )^{l_j} \times (1+ \sum_{j=1}^n r_j )  \right] \\
=& \widehat{\e}_{(x_j)_{j\le n}}^\omega \left[ \prod_{j=1}^n (1- p^\omega_{n-j}\ast \nu (x_{j-1} ) )^{l_j} \right] \nonumber\\
&+ \sum_{j=1}^n \widehat{\e}_{(x_j)_{j\le n}}^\omega \left[ \prod_{i\neq j,1\le i\le n} (1- p^\omega_{n-i}\ast \nu (x_{i-1} ) )^{l_i} \right] \widehat{\e}_{(x_j)_{j\le n}}^\omega \left[ (1- p^\omega_{n-j}\ast \nu (x_{j-1} ) )^{l_j} \times r_j\right]. \nonumber
\end{align}
%In any situation, we can show that there exists some positive constant $C>0$ such that for any $\delta\in(0,1]$, 
%\begin{equation}\label{bdlr}
%\widehat{\e}_{(x_j)_{j\le n}}^\omega \left[(1-\delta)^{l_j} r_j \right] \ge C \widehat{\e}_{(x_j)_{j\le n}}^\omega \left[(1-\delta)^{l_j} \right] \widehat{\e}_{(x_j)_{j\le n}}^\omega \left[ r_j \right].
%\end{equation}
%In fact, if $\mu_\bullet = \delta_1$ and $\omega(x_{j-1})=0$, $(l_j,r_j)$ is trivially zero vector by \eqref{lawlr}. So, both sides of the inequality \eqref{bdlr} are zero. Otherwise, $(l_j,r_j)$ is not trivially zero vector and we deduce from \eqref{lawlr} that
%\begin{align*}
%\widehat{\e}_{(x_j)_{j\le n}}^\omega \left[(1-\delta)^{l_j} r_j \right] \ge & \widehat{\e}_{(x_j)_{j\le n}}^\omega \left[\ind{l_j=0} r_j \right] = 
%\begin{cases}
%\sum_{m\ge0} m \mu_\circ(m+1) = \mu_\circ(0), & \textrm{ if } \omega(x_{j-1} ) =1;\\
%\sum_{m\ge0} m \mu_\bullet(m+1) = \mu_\bullet(0), & \textrm{ if } \omega(x_{j-1} ) =0.
%\end{cases}
%\end{align*}
%On the other hand,
%\begin{align*}
%\widehat{\e}_{(x_j)_{j\le n}}^\omega \left[(1-\delta)^{l_j} \right] \widehat{\e}_{(x_j)_{j\le n}}^\omega \left[ r_j \right] \le \widehat{\e}_{(x_j)_{j\le n}}^\omega \left[ r_j \right] =
%\begin{cases}
%\frac12\sigma^2_\circ, & \textrm{ if } \omega(x_{j-1})=1;\\
%\frac12\sigma^2_\bullet, & \textrm{ if } \omega(x_{j-1})=0.
%\end{cases}
%\end{align*}
%So, we can always find a constant $C>0$ such that $\mu_\circ(0)\ge \frac{C}{2}\sigma^2_\circ$ and $\mu_\bullet(0)\ge \frac{C}{2}\sigma^2_\bullet$. This leads to \eqref{bdlr}. 
Plugging \eqref{bdlr} into \eqref{mean-spinelr} yields that
\begin{align*}
&\widehat{\e}_{(x_j)_{j\le n}}^\omega \left[ \prod_{j=1}^n (1- p^\omega_{n-j}\ast \nu (x_{j-1} ) )^{l_j} \times (1+ \sum_{j=1}^n r_j )  \right] \\
\ge &  \widehat{\e}_{(x_j)_{j\le n}}^\omega \left[ \prod_{j=1}^n (1- p^\omega_{n-j}\ast \nu (x_{j-1} ) )^{l_j} \right] \left(1+C_1 \frac{1}{2}\sigma_\circ^2 \sum_{j=1}^{n}\omega(x_{j-1}) + C_1\frac12\sigma^2_\bullet \sum_{j=1}^{n}[1-\omega(x_{j-1})] \right)\\
\ge & \widehat{\e}_{(x_j)_{j\le n}}^\omega \left[ \prod_{j=1}^n (1- p^\omega_{n-j}\ast \nu (x_{j-1} ) )^{l_j} \right]  \frac{C_1}{2}\sigma_\circ^2 \sum_{j=1}^{n}\omega(x_{j-1}).
\end{align*}
Using it in \eqref{mean-spineS} yields that
\begin{align*}
1\ge &   \sum_{(x_j)_{j\le n} \in x+\mathscr{S}_n} \widehat{\p}^\omega_x( (S_{\zeta_j})_{j\le n} = (x_j)_{j\le n}) \widehat{\e}_{(x_j)_{j\le n}}^\omega \left[ \prod_{j=1}^n (1- p^\omega_{n-j}\ast \nu (x_{j-1} ) )^{l_j} \right] \frac{C_1}{2}\sigma_\circ^2 \sum_{j=1}^{n}\omega(x_{j-1})\\
= &\frac{C_1}{2} \sigma^2_\circ \widehat{\e}^\omega_x\left[ \prod_{j=1}^n (1- p^\omega_{n-j}\ast \nu (S_{\zeta_{j-1}} ) )^{l_j} \times \#I_\circ(n) \right]
\end{align*}
where $I_\circ(n)=\{j\in\N\cap[1,n]: \omega(S_{\zeta_{j-1}})=1\}$. Note that $\#I_\circ(n)$ has the same law as $T_n$ studied in Section \ref{RWRS}. For arbitrarily small $\varepsilon>0$, one sees that
\begin{align*}
1\ge & \frac{C_1}{2} \sigma^2_\circ \widehat{\e}^\omega_x\left[ \prod_{j=1}^n (1- p^\omega_{n-j}\ast \nu (S_{\zeta_{j-1}} ) )^{l_j} \times \#I_\circ(n) [\ind{\#I_\circ(n) > (p-\varepsilon) n} + \ind{\#I_\circ(n) \le (p-\varepsilon) n} ]\right] \\
\ge & \frac{C_1}{2} \sigma^2_\circ  \times (p-\varepsilon)n \times \widehat{\e}^\omega_x\left[ \prod_{j=1}^n (1- p^\omega_{n-j}\ast \nu (S_{\zeta_{j-1}} ) )^{l_j} \ind{\#I_\circ(n) > (p-\varepsilon) n} \right].
\end{align*}
Meanwhile, note that
\begin{align*}
p^\omega_n(x) = & \widehat{\e}^\omega_x\left[ \prod_{j=1}^n (1- p^\omega_{n-j}\ast \nu (S_{\zeta_{j-1}} ) )^{l_j} \right] \\
\le & \widehat{\e}^\omega_x\left[ \prod_{j=1}^n (1- p^\omega_{n-j}\ast \nu (S_{\zeta_{j-1}} ) )^{l_j} \ind{\#I_\circ(n) > (p-\varepsilon) n} \right] + \widehat{\p}^\omega_x\left(\#I_\circ(n) \le (p-\varepsilon) n \right) .
\end{align*}
Comparing these two inequalities and using \eqref{RWRS-LDP}, we deduce that  $\P_p$-a.s.,
\[
\sup_{|x|_1\le n^2 } p^\omega_n(x)  \le  \frac{1}{ \frac{C_1}{2} \sigma^2_\circ  \times (p-\varepsilon)n} + c_0^\omega(\varepsilon) e^{-\frac14 c_0(\varepsilon) n^{\frac{d}{d+2}}}.
\]
%Combining it with \eqref{probab-spine} and \eqref{RWRS-LDP}, we then deduce that for any $x\in\Z^d$ such that $|x|_1\le n^2$, $\P_p$-a.s.,
%\begin{align*}
%p^\omega_n(x) = & \widehat{\e}^\omega_x\left[ \prod_{j=1}^n (1- p^\omega_{n-j}\ast \nu (S_{\zeta_{j-1}} ) )^{l_j} \right] \\
%\le & \widehat{\e}^\omega_x\left[ \prod_{j=1}^n (1- p^\omega_{n-j}\ast \nu (S_{\zeta_{j-1}} ) )^{l_j} \ind{\#I_\circ(n) > (p-\varepsilon) n} \right] + \widehat{\p}^\omega_x\left(\#I_\circ(n) \le (p-\varepsilon) n \right) \\
%\le & \frac{1}{ \frac{C}{2} \sigma^2_\circ  \times (p-\varepsilon)n} + C_\omega(\varepsilon) e^{-\frac14 C'_\varepsilon n^{\frac{d}{d+2}}}.
%\end{align*}
This gives us an upper bound for $p^\omega_n(x)$: $\P_p$-a.s., 
\begin{equation}\label{upbdp}
\max_{|x|_1\le n^2} p_n^\omega(x) \le \frac{ c_3^\omega }{n}, \forall n\ge 1.
\end{equation}

Next, let us turn to study the limit of $m p_m^\omega(x)$ for $|x|_1\le n$ and $an\le m\le bn$ with $0<a<b<\infty$ fixed.

\begin{proof}[ Proof of Theorem \ref{thm: UKprobab}]
In view of  \eqref{mean-spineS}, \eqref{probab-spine} and  \eqref{mean-spinelr}, we see that 
\begin{align}\label{mean-spineR}
1= p_n^\omega(x) + &\sum_{(x_j)_{j\le n} \in x+\mathscr{S}_n} \widehat{\p}^\omega_x( (S_{\zeta_j})_{j\le n} = (x_j)_{j\le n})  \nonumber   \\
&\times \sum_{j=1}^n \widehat{\e}_{(x_j)_{j\le n}}^\omega \left[ \prod_{i\neq j,1\le i\le n} (1- p^\omega_{n-i}\ast \nu (x_{i-1} ) )^{l_i} \right] \widehat{\e}_{(x_j)_{j\le n}}^\omega \left[ (1- p^\omega_{n-j}\ast \nu (x_{j-1} ) )^{l_j} \times r_j\right].
\end{align}
Observe that for the admissible simple random walk path $(x=x_0, x_1,\cdots, x_n)$, $\sup_{0\le k\le n}| x-x_k|_1\le n$. It is known from \eqref{upbdp} that $p^\omega_{n-j}\ast \nu (x_{j-1}) $ is uniformly $o_n(1)$ for $n-j\ge q_n := \ceil{n^{4/5}}$ and $|x|_1\le n^{3/2}\ll q_n^2$. %So, let us consider $\widehat{\e}_{(x_j)_{j\le n}}^\omega \left[(1-\delta)^{l_j} r_j \right] $ for $\delta\downarrow 0+$.

For $n\gg1$, applying \eqref{asymplr} for $j\le n-q_n$ and applying  \eqref{bdlr} for $n-q_n< j\le n$ in \eqref{mean-spineR} yields that
\begin{align*}
%&\widehat{\e}_{(x_j)_{j\le n}}^\omega \left[ \prod_{j=1}^n (1- p^\omega_{n-j}\ast \nu (x_{j-1} ) )^{l_j} \times (1+ \sum_{j=1}^n r_j )  \right] \\
1\le  p_n^\omega(x) + &\sum_{(x_j)_{j\le n} \in x+\mathscr{S}_n} \widehat{\p}^\omega_x( (S_{\zeta_j})_{j\le n} = (x_j)_{j\le n}) \times\widehat{\e}_{(x_j)_{j\le n}}^\omega \left[ \prod_{i=1}^n (1- p^\omega_{n-i}\ast \nu (x_{i-1} ) )^{l_i} \right]\\
&\times \left(  \sum_{1\le j\le n-q_n} \left( \widehat{\e}_{(x_j)_{j\le n}}^\omega \left[ r_j \right]+o_n(1)\right) +   C_2 \sum_{n-q_n< j\le n} \widehat{\e}_{(x_j)_{j\le n}}^\omega \left[ r_j \right]\right),
\end{align*}
and that
\begin{align*}
1\ge p_n^\omega(x) + &\sum_{x=x_0,x_1,\cdots, x_n\in\Z^d} \widehat{\p}^\omega_x( S_{\zeta_j} = x_j, \forall 1\le j\le n) \times\widehat{\e}_{(x_j)_{j\le n}}^\omega \left[ \prod_{i=1}^n (1- p^\omega_{n-i}\ast \nu (x_{i-1} ) )^{l_i} \right]\\
&\times \left(  \sum_{1\le j\le n-q_n} \left( \widehat{\e}_{(x_j)_{j\le n}}^\omega \left[ r_j \right] +o_n(1)\right) +   C_1 \sum_{n-q_n< j\le n} \widehat{\e}_{(x_j)_{j\le n}}^\omega \left[ r_j \right]\right)
\end{align*}
Note that 
\[
\sum_{j=1}^n  \widehat{\e}_{(x_j)_{j\le n}}^\omega \left[ r_j \right]  =  \frac12\sigma^2_\circ \sum_{j=1}^n \ind{\omega(x_{j-1})=1}+ \frac12\sigma^2_\bullet \sum_{j=1}^n\ind{\omega(x_{j-1})=0} .
\]
and that $ \sum_{n-q_n< j\le n} \widehat{\e}_{(x_j)_{j\le n}}^\omega \left[ r_j \right] + \sum_{1\le j\le n-q_n} o_n(1)= o(n)$. Recall \eqref{probab-spine} and set $\overline{I}(n):=\sigma^2_\circ \sum_{j=1}^n\ind{\omega(S_{\zeta_{j-1}})=1} + \sigma^2_\bullet \sum_{j=1}^n \ind{\omega(S_{\zeta_{j-1}})=0}$, we therefore obtain that
\begin{align*}
&|1-p_n^\omega(x) -  \widehat{\e}^\omega_x \left[ \prod_{i=1}^n (1- p^\omega_{n-i}\ast \nu (S_{\zeta_{i-1}} ) )^{l_i} \frac{\overline{I}(n)}{2} \right] |  
\le  o(n) p_n^\omega(x).
\end{align*}
By the upper bound \eqref{upbdp}, we get that uniformly for $|x|_1\le n^{3/2}$,
\begin{align}\label{mean-keyid}
&\widehat{\e}^\omega_x \left[ \prod_{i=1}^n (1- p^\omega_{n-i}\ast \nu (S_{\zeta_{i-1}} ) )^{l_i} \frac{\overline{I}(n)}{2} \right] 
= 1+o_n^\omega(1). 
\end{align}
 Observe that
\begin{equation*}
\overline{I}(n)= \sigma^2_\circ \sum_{j=1}^n \omega(S_{\zeta_{j-1}})+ \sigma^2_\bullet [n-\sum_{j=1}^n \omega(S_{\zeta_{j-1}})] = (\sigma^2_\circ -\sigma^2_\bullet) \# I_\circ(n) + \sigma^2_\bullet n.
\end{equation*}
In view of \eqref{asym-spine}, $\widehat{\p}^\omega_x-$a.s., $\frac{\overline{I}(n)}{n} \to p\sigma^2_\circ+ (1-p)\sigma^2_\bullet$. In general, this convergence cannot be applied directly here\footnote{ Unless the special case where $\sigma_\circ^2 = \sigma_\bullet^2$.}. We will replace $\overline{I}(n)$ by $n(p\sigma^2_\circ+ (1-p)\sigma^2_\bullet)$ in the expectation in \eqref{mean-keyid} and control the error term.  We write for simplicity that $ \overline{\sigma}^2:=  p \sigma^2_\circ+ (1-p)\sigma^2_\bullet$.
Then \eqref{mean-keyid} becomes
\begin{align}\label{mean-spine-probab}
1+o^\omega_n(1)% = & \widehat{\e}^\omega_x \left[ \prod_{i=1}^n (1- p^\omega_{n-i}\ast \nu (S_{\zeta_{i-1}} ) )^{l_i} \frac{\overline{\sigma}^2}{2}n \right]  +  \widehat{\e}^\omega_x \left[ \prod_{i=1}^n (1- p^\omega_{n-i}\ast \nu (S_{\zeta_{i-1}} ) )^{l_i} \left( \frac{\overline{I}(n) }{2} - \frac{\overline{\sigma}^2}{2}n \right)\right] \\
=& \frac{\overline{\sigma}^2}{2}n \times \widehat{\e}^\omega_x \left[ \prod_{i=1}^n (1- p^\omega_{n-i}\ast \nu (S_{\zeta_{i-1}} ) )^{l_i} \right] + Error_\eqref{mean-spine-probab},
%=&  \frac{\overline{\sigma}^2}{2}n \times p^\omega_n(x) + Error_\eqref{mean-spine-probab}, \nonumber
\end{align}
where 
\[
Error_\eqref{mean-spine-probab} : = \widehat{\e}^\omega_x \left[ \prod_{i=1}^n (1- p^\omega_{n-i}\ast \nu (S_{\zeta_{i-1}} ) )^{l_i} \left( \frac{\overline{I}(n) }{2} - \frac{\overline{\sigma}^2}{2}n \right)\right] .
\]
By \eqref{probab-spine}, one sees that 
\begin{equation}\label{mean-probab}
1+o^\omega_n(1)= \frac{\overline{\sigma}^2}{2}n \times p^\omega_n(x) + Error_\eqref{mean-spine-probab} .
\end{equation}
It remains to deal with $Error_\eqref{mean-spine-probab} $. Observe that for arbitrarily small $\varepsilon>0$,  
\begin{align*}
& | Error_\eqref{mean-spine-probab}  | \le   \widehat{\e}^\omega_x \left[ \prod_{i=1}^n (1- p^\omega_{n-i}\ast \nu (S_{\zeta_{i-1}} ) )^{l_i} \bigg\vert   \frac{\overline{I}(n) }{2} - \frac{\overline{\sigma}^2}{2}n \bigg\vert \right]\\
\le &   \widehat{\e}^\omega_x \left[ \prod_{i=1}^n (1- p^\omega_{n-i}\ast \nu (S_{\zeta_{i-1}} ) )^{l_i} \bigg\vert   \frac{\overline{I}(n) }{2} - \frac{\overline{\sigma}^2}{2}n \bigg\vert (\ind{| \frac{\overline{I}(n) }{2}- \frac{\overline{\sigma}^2}{2} n|\le \varepsilon n }+ \ind{|\frac{\overline{I}(n) }{2}- \frac{\overline{\sigma}^2}{2} n|> \varepsilon n }) \right] \\
\le & \varepsilon n   \widehat{\e}^\omega_x \left[ \prod_{i=1}^n (1- p^\omega_{n-i}\ast \nu (S_{\zeta_{i-1}} ) )^{l_i} \right] + \widehat{\e}^\omega_x\left[ \bigg\vert   \frac{\overline{I}(n) }{2} - \frac{\overline{\sigma}^2}{2}n \bigg\vert \ind{|\frac{\overline{I}(n) }{2}- \frac{\overline{\sigma}^2}{2} n|> \varepsilon n }\right].
\end{align*}
On the one hand, by \eqref{probab-spine} and \eqref{upbdp}, uniformly for $|x|_1\le n^{3/2}$,
\[
 \varepsilon n   \widehat{\e}^\omega_x \left[ \prod_{i=1}^n (1- p^\omega_{n-i}\ast \nu (S_{\zeta_{i-1}} ) )^{l_i} \right] \le \varepsilon c_3^\omega.
\]
On the other hand, as $ \overline{I}(n) \le (\sigma^2_\circ + \sigma^2_\bullet ) n$,  we have
\begin{align*}
&\widehat{\e}^\omega_x\left[ \bigg\vert   \frac{\overline{I}(n) }{2} - \frac{\overline{\sigma}^2}{2}n \bigg\vert \ind{|\frac{\overline{I}(n) }{2}- \frac{\overline{\sigma}^2}{2} n|> \varepsilon n }\right]
\le  (\sigma^2_\circ + \sigma^2_\bullet ) n \widehat{\p}^\omega_x\left( |\frac{\overline{I}(n) }{2}- \frac{\overline{\sigma}^2}{2} n|> \varepsilon n \right) \\
&=   (\sigma^2_\circ + \sigma^2_\bullet ) n  \widehat{\p}^\omega_x\left( | (\sigma_\circ^2-\sigma_\bullet^2)\#I_\circ(n) -(\sigma_\circ^2-\sigma_\bullet^2) p n | > 2\varepsilon n \right)= o_n^\omega(1),
\end{align*}
 uniformly for $|x|_1\le n^{3/2}$ by \eqref{RWRS-LDP}, because $\#I_\circ(n)$ has the same distribution as $T_n$ under $\Q^\omega$. As a consequence, we obtain that
\[
 \limsup_{n\to\infty}\sup_{|x|_1\le n^{3/2}}
 |\operatorname{Error}_{\eqref{mean-spine-probab}}|
 \le c_3^\omega\varepsilon.
\]
Going back to \eqref{mean-probab}, we thus deduce that for $\P_p$-a.s. $\omega$,  uniformly for $|x|_1\le n^{3/2}$,
\[
p^\omega_n(x) = \frac{1+o^\omega_n(1) -\operatorname{Error}_{\eqref{mean-spine-probab}}}{\frac{ \overline{\sigma}^2}{2} n}.
\]
This suffices to conclude that
\[
\lim_{n\to\infty} \sup_{|x|_1\le n}\sup_{an \le m \le bn} \bigg\vert m p^\omega_m(x) - \frac{2}{p\sigma^2_\circ + (1-p)\sigma^2_\bullet} \bigg\vert =0.
\]
\end{proof}

\subsection{Quenched Yaglom theorem: proof of Theorem \ref{thm: Yaglom}}

In this section, we establish the quenched Yaglom theorem via Stein method developed in \cite{PR}. In fact, Theorem 2.1 of \cite{PR} says that if $W$ is a non-negative random variable with finite variance under $\P$, then  the Wasserstein distance between the law of $W$ and the exponential law with parameter $1$ satisfies
%then the Wasserstein distance between the law of $W$ and the exponential law with parameter $1$ is bounded by $2\E[|W^e-W|]$ where $W^e$ has the equilibrium law w.r.t. $W$. 
\begin{equation}\label{steinbd}
d_W(\mathscr{L}(\frac{W}{\E[W]},\P), Exp(1) ) \le \frac{2}{\E[W]} \E[|W^e-W|],
\end{equation}
where $d_W(\mu,\nu):=\sup\{|\int f d\mu - \int f d\nu|: f: \R\to\R \textrm{ is Lipschitz and } \| f'\| \le 1\}$ and $\P(W^e \le x) = \frac{1}{\E[W]}\int_0^x \P(W>y)dy$. Note that the convergence in the Wasserstein distance implies the weak convergence. So, we only need to prove that the Wasserstein distance of the law of $\frac{Z_n}{\overline{\sigma}^2 n/2}$ under $\p^\omega_x(\cdot\vert Z_n \ge 1)$ and the exponential law with parameter $1$ converges to zero.

Nevertheless, the arguments in Section 3.3 of \cite{PR} for the classical Galton-Watson process,  can not be applied directly in our setting because of the inhomogeneous branching laws in the random environment $\omega$. New ideas are developed to construct the coupling of $W^e$ and $W$. More precisely, we will construct a r.v. $W_n$ such that
\[
\mathscr{L}(W_n, \widehat{\p}^\omega_x) = \mathscr{L}(Z_n, \p^\omega_x(\cdot\vert Z_n \ge 1)),
\]
and another r.v. $W_n^e$ such that $W_n^e$ has equilibrium law w.r.t. $W_n$ under $\widehat{\p}^\omega_x$.

\subsubsection{Finite second moment for $W_n$}

In this part, we check that for $\P_p$-a.s. environment $\omega$,
\begin{equation}\label{cond2mombd}
\sup_{n\ge1}\e^\omega\left[ \left( \frac{Z_n}{\overline{\sigma}^2 n/2} \right)^2 \vert Z_n \ge 1 \right] <\infty.
\end{equation}

First, it follows from \eqref{lwbdp} that 
\begin{align}\label{upbdcondm}
\e^\omega[Z_n\vert Z_n \ge 1] = & \frac{ \e^\omega[Z_n\ind{Z_n \ge 1 }] }{ \p^\omega(Z_n \ge 1 )} = \frac{1}{\p^\omega(Z_n \ge 1 ) } \le (C_0+1)n, \forall n\ge 1.
\end{align}
By Theorem \ref{thm: UKprobab}, one obtains that
\begin{equation}\label{asymcondm}
\e^\omega[Z_n\vert Z_n \ge 1] \sim \frac{p\sigma^2_\circ+ (1-p)\sigma^2_\bullet}{2}n = \frac{\overline{\sigma}^2}{2} n.
\end{equation} 
Next, for the second moment, by \eqref{2mombd} and \eqref{lwbdp}, one sees that
\begin{align*}
\e^\omega[Z_n^2 \vert Z_n \ge 1] = \frac{\e^\omega[Z_n^2]}{ \p^\omega(Z_n \ge 1 ) } \le (C_0 n +1)^2,
\end{align*}
which is sufficient to conclude \eqref{cond2mombd}.
\subsubsection{Coupling under $\widehat{\p}^\omega_x$}

In fact, we only need to prove the Yaglom theorem under $\p^\omega_0$, as for any $x\in\Z^d$, the shifted environment $(\theta_x\omega)(z)=\omega(z+x)$ has the same law $\P_p$. So, it suffices to make the coupling under $\widehat{\p}^\omega$.

We start with constructing the r.v. $W_n^e$ with equilibrium law under $\widehat{\p}_x^\omega$.  If $W$ is a non-negative random variable with finite mean under $\P$, let $W^+$ be a r.v. distributed as $W$ conditioned on $W>0$, there is a way to obtain a r.v. of the equilibrium law w.r.t. $W^+$ as follow:
\begin{itemize}
\item First, take $W^{sb}$ to be a r.v. with the size-biased law of $W$, i.e., $\P(W^{sb}\in \cdot ) = \frac{1}{\E[W]}\E[W\ind{W\in\cdot}]$. Then $W^{sb} \egloi (W^+)^{sb}$.
\item Secondly, take an independent r.v. $U$ which is uniformly distributed in $(0,1)$. Then $UW^{sb}\egloi (W^+)^e$ has the equilibrium law w.r.t. $W^+$.
\end{itemize}

In view of \eqref{prop-spine-1}, $\mathscr{L}(Z_n, \widehat{\p}^\omega_x)$ has the size-biased law of $\mathscr{L}(Z_n, \p^\omega_x(\cdot\vert Z_n \ge 1))$. As $\zeta_n$ is uniformly chosen from $Z_n$ individuals of the $n$-th generation, $R_n-U$ has the same law as $UZ_n$ under $\widehat{\p}^\omega_x$ with $U\sim U(0,1)$ a r.v. independent of all others. Thus, we get $W_n^e=R_n-U$ which has equilibrium law w.r.t. $\mathscr{L}(Z_n, \p^\omega_x(\cdot\vert Z_n \ge 1))$.

Next, we need to construct a r.v. $W_n$ which, under $\widehat{\p}^\omega_x$, is distributed as $\mathscr{L}(Z_n, \p^\omega_x(\cdot\vert Z_n \ge 1))$. For any continuous and bounded function $f:\R\to\R$, it follows from \eqref{keyeq} that
\begin{align*}
&\e^\omega_x[ f(Z_n) \vert Z_n \ge 1 ] =  \frac{1}{\p^\omega_x(Z_n \ge 1)} \e^\omega_x[ f(Z_n) \ind{Z_n\ge 1} ] = \frac{1}{\widehat{\p}^\omega_x(L_n=0)} \widehat{\e}^\omega_x [f(Z_n)\ind{L_n=0}]\\
=& \frac{1}{\widehat{\p}^\omega_x(L_n=0)}  \widehat{\e}^\omega_x [f(R_n)\ind{L_n=0}]\\
=& \sum_{(x_j)_{j\le n}\in x+\mathscr{S}_n} \widehat{\p}^\omega_x((S_{\zeta_j})_{j\le n} =( x_j)_{j\le n} \vert L_n =0 ) \widehat{\e}^\omega_x [f(R_n)\vert L_n=0, (S_{\zeta_j})_{j\le n} =( x_j)_{j\le n} ].
\end{align*}
First, let $(\Xi_{j}^{(n)}; 0\le j\le n)$ be a random vector such that
\[
\widehat{\p}^\omega_x\left( (\Xi_j^{(n)})_{0\le j\le n}\in \cdot \right ) = \widehat{\p}^\omega_x((S_{\zeta_j})_{0\le j\le n} \in \cdot \vert L_n =0 ).
\]
Secondly, as we can take $f(t) = e^{-\lambda t}$ with $\lambda>0$, note that
\begin{align*}
 &\widehat{\e}^\omega_x [e^{-\lambda R_n} \vert L_n=0, (S_{\zeta_j})_{j\le n} =( x_j)_{j\le n} ]\\
 =& \frac{ \widehat{\e}^\omega_x \left[e^{-\lambda R_n} \ind{ L_n=0, (S_{\zeta_j})_{j\le n} =( x_j)_{j\le n} } \right] }{ \widehat{\p}^\omega_x(  L_n=0, (S_{\zeta_j})_{j\le n} =( x_j)_{j\le n}  )} 
 = \frac{ \widehat{\e}^\omega_{(x_j)_{j\le n}} \left[e^{-\lambda R_n} \ind{ L_n=0} \right] }{\widehat{\e}^\omega_{(x_j)_{j\le n}} \left[\ind{L_n=0}\right]} ,
\end{align*}
which by independence of $(l_j,r_j)$ implies that 
\begin{align*}
 \widehat{\e}^\omega_x [e^{-\lambda R_n} \vert L_n=0, (S_{\zeta_j})_{j\le n} =( x_j)_{j\le n} ]=&e^{-\lambda} \frac{\prod_{j=1}^n \widehat{\e}^\omega_{(x_j)_{j\le n}} \left[ e^{-\lambda R_{n,j}}\ind{L_{n,j}=0} \right] }{ \prod_{j=1}^n  \widehat{\e}^\omega_{(x_j)_{j\le n}} \left[ \ind{L_{n,j}=0} \right]} \\
 =& e^{-\lambda} \prod_{j=1}^n  \widehat{\e}^\omega_{(x_j)_{j\le n}} \left[ e^{-\lambda R_{n,j}}\vert L_{n,j}=0 \right].
\end{align*}
Here, given $(\Xi_j^{(n)})_{0\le j\le n} = (x_j)_{0\le j\le n}$, for any $1\le j\le n$, we introduce a new random variable $R^*_{n,j}$ which, under $\widehat{\p}^\omega_{(x_j)_{j\le n}}$, is distributed as $R_{n,j}$ under $\widehat{\p}^\omega_{(x_j)_{j\le n}}(\cdot \vert L_{n,j}=0)$ as follows: 
\begin{itemize}
\item let $\{R'_{n,j}\}_{1\le j\le n}$ be a family of independent random variables, which is independent of all other objects, so that $R'_{n,j}$ is distributed as $R_{n,j}$ under $\widehat{\p}^\omega_{(x_j)_{j\le n}}(\cdot \vert L_{n,j}=0)$;
\item let
\[
R^*_{n,j}:= R_{n,j}\ind{L_{n,j}=0} + R'_{n,j}\ind{L_{n,j}\ge 1}.
\]
\end{itemize}
Set $R^*_n:= 1+ \sum_{j=1}^n R^*_{n,j}$. Then,
\begin{align*}
\e^\omega_x[e^{-\lambda Z_n } \vert Z_n \ge 1 ] = &  \sum_{(x_j)_{j\le n}\in x+\mathscr{S}_n} \widehat{\p}^\omega_x((\Xi_j^{(n)})_{j\le n} =( x_j)_{j\le n}  ) \widehat{\e}^\omega_{(x_j)_{j\le n}}[e^{-\lambda R_n^*}] .%= \widehat{\e}_x^\omega[e^{-\lambda R_n^*}].
\end{align*}

Now, we need to couple $R_n-U$ and $R_n^*$ so that we could control $\widehat{\e}^\omega_x[|R_n-U - R_n^*|]$. We first couple the motions of the spine $(S_{\zeta_j})_{j\le n}$ and $(\Xi_j^{(n)})_{j\le n}$. Take a fixed and small $\delta\in(0,1)$ and set $m_n=\floor{m_n}$ and $k_n=n-m_n$, let us compare the laws of $(S_{\zeta_j})_{j\le m_n}$ and $(\Xi_j^{(n)})_{j\le m_n}$ in the following lemma.

\begin{lem}\label{couple-spine}
For any fixed $\delta\in(0,1)$, for $\P_p$-a.s. environment $\omega$, we have
\[
\lim_{n\to\infty} \sum_{(x_j)_{j\le m_n} \in \mathscr{S}_{m_n}} |\widehat{\p}^\omega( (\Xi_j^{(n)})_{j\le m_n} = (x_j)_{j\le m_n} ) - \P( (S_j)_{j\le m_n} = (x_j)_{j\le m_n}) | =0,
\]
where $\mathscr{S}_m=\{(x_j)_{j\le m} \in (\Z^d)^{m+1}: x_0=0, |x_j-x_{j-1}|_1=1, \forall 1\le j \le m\}$ is the collection of simple random walk path with length $m$.
\end{lem}

The proof of Lemma \ref{couple-spine} is postponed to Section \ref{lems}. By admitting it now, for any fixed $\delta\in(0,1)$, it follows from the optimal coupling property of the total variation distance and the gluing lemma that we could couple $(S_j)_{j\le n}$ and $(\Xi_j^{(n)})_{j\le n}$ in the same probability space so that $\P_p$-a.s. 
\begin{equation}\label{TV-coupling}
\widehat{\p}^\omega\left( \cup_{0\le j\le m_n} \{\Xi_j^{(n)} \neq S_j \} \right) \le \varepsilon_n(\delta, \omega),
\end{equation}
with $\varepsilon_n(\delta,\omega)\to 0$  as $n\to\infty$. 

\begin{figure}[H]
    \centering

    \makebox[\textwidth][c]{%
        \begin{subfigure}[b]{0.54\textwidth}
            \centering
            \includegraphics[width=\linewidth]{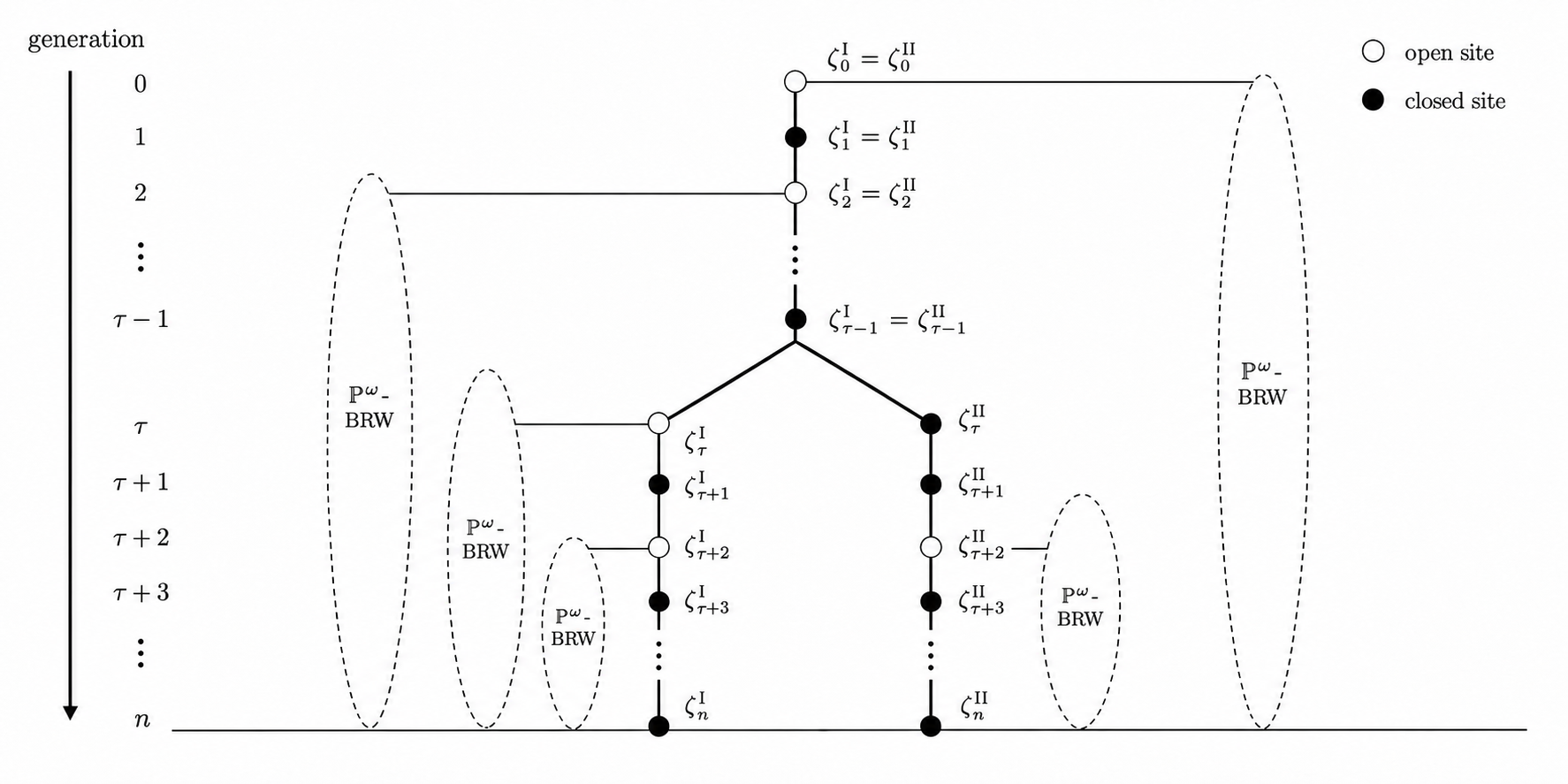}
            \caption{Genealogical representation.}
            \label{fig:genealogical}
        \end{subfigure}%
        \hspace{-0.04\textwidth}%
        \begin{subfigure}[b]{0.54\textwidth}
            \centering
            \includegraphics[
                width=\linewidth,
                trim={0 110bp 0 0},
                clip
            ]{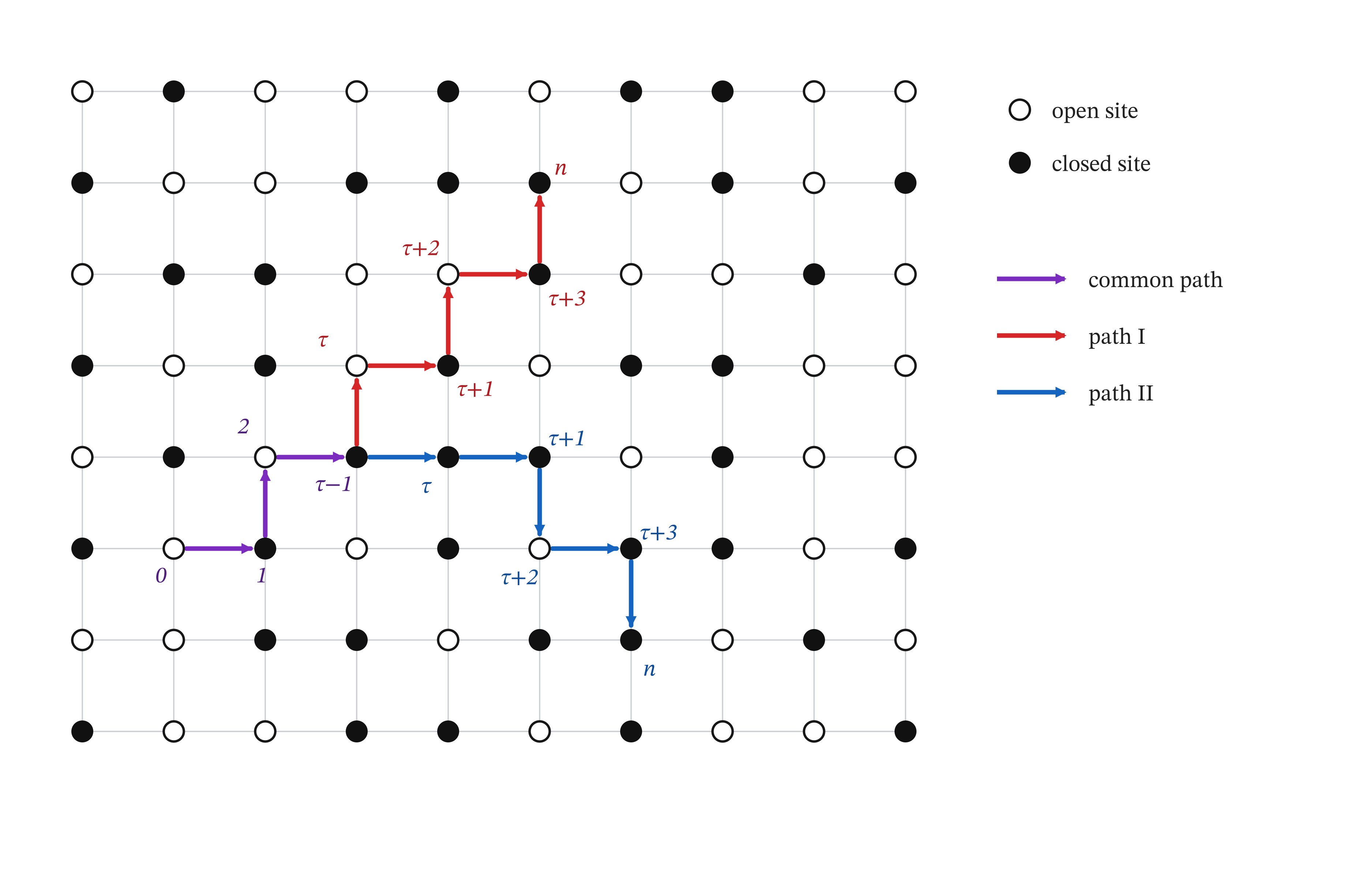}
            \caption{Spatial representation of two spines.}
            \label{fig:spatial}
        \end{subfigure}%
    }

    \caption{The system with 2 spines.}
    \label{fig:two-spine-comparison}
\end{figure}
Now we are ready to construct a system with 2 spines to realize the coupling of $R_n-U$ and $R_n^*$, see Fig \ref{fig:two-spine-comparison}. 
\begin{itemize}
\item Given the environment $\omega$, we sample two paths $(\xi_j^{I}, \xi_j^{I\!I})_{j\le n}$ according to the joint law of $(\Xi_j^{(n)}, S_j)_{j\le n}$ which correspond to the motion of two spines with one common ancestor, denoted by $(\zeta_j^{I})_{j\le n}$ and $(\zeta_j^{I\!I})_{j\le n}$, respectively. Note that $\zeta_0^I=\zeta_0^{I\!I}$ and $\xi_0^I=\xi_0^{I\!I}=0\in\Z^d$. $\zeta_{j+1}^I$ is viewed as a child of $\zeta_j^I$ while $\zeta_{j+1}^{I\!I}$ is viewed as a child of $\zeta_j^{I\!I}$. Up to the first splitting time of the two paths $\tau:=\inf\{ m \ge 1: \xi_m^I\neq \xi_m^{I\!I}\}$, we view $\zeta_j^I=\zeta_j^{I\!I}$ for all $j<\tau$. 
\item Given the paths of the spines $(\xi_j^{I}, \xi_j^{I\!I})_{j\le n}=(x_j,y_j)_{j\le n}$, for $j<\tau-1$, $\zeta_j^I=\zeta_j^{I\!I}$ and to $\zeta_j^I$ we attach $l_{j+1}$ children on the left of $\zeta^I_{j+1}$ and attach $r_{j+1}$ children on the right of $\zeta^I_{j+1}$ so that $(l_{j+1}, r_{j+1})$ are independent and distributed as \eqref{lawlr}; for $\tau-1\le j\le n-1$, $\zeta_j^I$ and $\zeta_j^{I\!I}$ independently produce their children so that $l_{j+1}^I$ (and $r_{j+1}^I$) siblings are put on the left (and right respectively) of $\zeta^I_{j+1}$ and $l_{j+1}^{I\!I}$ (and $r_{j+1}^{I\!I}$) siblings are put on the left (and right respectively) of $\zeta_{j+1}^{I\!I}$. Each attached child makes independently a simple random walk jump from the position of its parent and then produces its own descendants as the usually branching random walk in the environment $\omega$. For conveniance, write $l_j^I=l_j^{I\!I}=l_j$ and $r_j^I=r_j^{I\!I}=r_j$ for $j\le \tau-1$.
\item At the $n$-th generation, we count the number of particles descending from the $l_j^I$ particles on the left (or $r_j^I$ on the right) of $\zeta_i^I$ and denote it by $L^I_{n,j}$ (or $R_{n,j}^I$ respectively). Similarly, we get $(L_{n,j}^{I\!I}, R_{n,j}^{I\!I})$. 
\end{itemize}
We thus get a model with 2 spines up to the $n$-th generation and denote it quenched law by $\widehat{\p}^{\omega, \textrm{2-spine}}$. The corresponding expectation is denoted by $\widehat{\e}^{\omega, \textrm{2-spine}}$. Apparently, if we take $R_n^{I\!I}:=1+\sum_{j=1}^n R_{n,j}^{I\!I}$ and $U$ an independent r.v. uniformly distributed in $(0,1)$, then under $\widehat{\p}^{\omega, \textrm{2-spine}}$, $R_n^{I\!I}-U$ has the equilibrium law w.r.t. $\mathscr{L}(Z_n, \p^\omega(\cdot\vert Z_n \ge1))$.

 Denote the quenched law conditioned on $(\xi_j^{I}, \xi_j^{I\!I})_{j\le n}=(x_j,y_j)_{j\le n}$ by $\widehat{\p}^{\omega, \textrm{2-spine}}_{(x_j)_{j\le n}, (y_j)_{j\le n}}$. Under $\widehat{\p}^{\omega, \textrm{2-spine}}_{(x_j)_{j\le n}, (y_j)_{j\le n}}$, we have the first splitting time $\tau_{(x_j,y_j)_{j\le n}}:=\inf\{ 1\le k\le n: x_k\neq y_k \}$ with the convention that $\inf\emptyset = n+1$. As stated above, we have constructed two spines $(\zeta_j^I, \zeta_j^{I\!I})_{j\le n}$. Considering the process associated with the first spine, under $\widehat{\p}^{\omega, \textrm{2-spine}}_{(x_j)_{j\le n}, (y_j)_{j\le n}}$, let $\widetilde{R}_{n,j}^I$ be an independent r.v. distributed as $R^I_{n,j}$ conditioned on $\{ L_{n,j}^I = 0\}$ and let
 \[
 R^{I,*}_{n,j}:= \widetilde{R}_{n,j}^I\ind{L^I_{n,j} \ge 1} + R^I_{n,j}\ind{ L^I_{n,j}=0 }, \forall 1\le j\le n.
 \]
Then, under $\widehat{\p}^{\omega, \textrm{2-spine}}$, $R^{I,*}_{n}:= 1+ \sum_{j=1}^n R^{I,*}_{n,j}$ is distributed as $\mathscr{L}(Z_n, \p^\omega(\cdot\vert Z_n \ge1))$. We also set $R_n^{I}:=1+\sum_{j=1}^n R_{n,j}^I$. We thus deduce from \eqref{steinbd} and \eqref{asymcondm} that
\begin{align}\label{dWbd}
& d_W( \mathscr{L}(\frac{Z_n}{ \e^\omega[Z_n \vert Z_n \ge 1] }, \p^\omega(\cdot\vert Z_n \ge1)), Exp(1) ) \le  \frac{2 \widehat{\e}^{\omega, \textrm{2-spine}}\left[ | R^{I,*}_n - (R^{I\!I}_n -U ) | \right] }{ \e^\omega[Z_n \vert Z_n \ge 1]  }\\
\le &  \frac{2}{ \e^\omega[Z_n \vert Z_n \ge 1]  } \left( 1+  \widehat{\e}^{\omega, \textrm{2-spine}}\left[ | R^{I,*}_n - R^{I\!I}_n| \right] \right)\nonumber\\
= & o_n^\omega(1)+ \frac{2}{\overline{\sigma}^2/2+o_n^\omega(1)}\left[\frac1n   \widehat{\e}^{\omega, \textrm{2-spine}}\left[ | R^{I,*}_n - R^{I}_n| \right] +  \frac1n   \widehat{\e}^{\omega, \textrm{2-spine}}\left[ | R^{I}_n - R^{I\!I}_n| \right]  \right].\nonumber
\end{align}
On the one hand, observe that
\begin{align*}
 & \widehat{\e}^{\omega, \textrm{2-spine}}\left[ | R^{I,*}_n - R^{I}_n| \right]  =   \widehat{\e}^{\omega, \textrm{2-spine}}\left[ | \sum_{j=1}^n (R^{I,*}_{n,j} - R^I_{n,j} ) | \right] \\
 & \le  \widehat{\e}^{\omega, \textrm{2-spine}}\left[ \sum_{j=1}^n  \widetilde{R}_{n,j}^I\ind{L^I_{n,j} \ge 1} +\sum_{j=1}^n  R^I_{n,j}\ind{ L^I_{n,j}\ge 1 } \right]\\
 &\le \sup_{(x_j)_{j\le n}, (y_j)_{j\le n}} \widehat{\e}_{(x_j)_{j\le n}, (y_j)_{j\le n}}^{\omega, \textrm{2-spine}}\left[ \sum_{j=1}^n  \widetilde{R}_{n,j}^I\ind{L^I_{n,j} \ge 1} +\sum_{j=1}^n  R^I_{n,j}\ind{ L^I_{n,j}\ge 1 } \right].
\end{align*}
Here we see that
\begin{align*}
 &\widehat{\e}_{(x_j)_{j\le n}, (y_j)_{j\le n}}^{\omega, \textrm{2-spine}}\left[ \sum_{j=1}^n  \widetilde{R}_{n,j}^I\ind{L^I_{n,j} \ge 1} +\sum_{j=1}^n  R^I_{n,j}\ind{ L^I_{n,j}\ge 1 } \right] \\
 &=\sum_{j=1}^n \left( \widehat{\e}^\omega_{(x_j)_{j\le n}}[R_{n,j}\vert L_{n,j}=0] \widehat{\p}^\omega_{(x_j)_{j\le n}}(L_{n,j}\ge 1) + \widehat{\e}^\omega_{(x_j)_{j\le n}}[ R_{n,j}\ind{L_{n,j}\ge 1} ] \right)\\
 &= \sum_{j=1}^n  \frac{ \widehat{\e}^\omega_{(x_j)_{j\le n}}[r_j(1-p^\omega_{n-j}\ast\nu(x_{j-1}))^{l_j} ] }{\widehat{\e}^\omega_{(x_j)_{j\le n}}[(1-p^\omega_{n-j}\ast\nu(x_{j-1}))^{l_j} ]} \widehat{\e}^\omega_{(x_j)_{j\le n}}[1-(1-p^\omega_{n-j}\ast\nu(x_{j-1}))^{l_j} ] \\
 & \qquad + \sum_{j=1}^n \widehat{\e}^\omega_{(x_j)_{j\le n}}\left[ r_j[1-(1-p^\omega_{n-j}\ast\nu(x_{j-1}))^{l_j} ] \right].
\end{align*}
By \eqref{bdlr} and the fact that $1-(1-\delta)^{l_j} \le 1\wedge (l_j\delta)$ for any $\delta\in[0,1]$, one obtains that 
\begin{align*}
& \widehat{\e}^{\omega, \textrm{2-spine}}\left[ | R^{I,*}_n - R^{I}_n| \right]  \\
\le & C_2 \sum_{j=1}^n \widehat{\e}^\omega_{(x_j)_{j\le n}}[r_j] \widehat{\e}^\omega_{(x_j)_{j\le n}}\left[ 1 \wedge (l_j p^\omega_{n-j}\ast\nu(x_{j-1}))\right] + \sum_{j=1}^n \widehat{\e}^\omega_{(x_j)_{j\le n}}\left[ r_j\left(1 \wedge (l_j p^\omega_{n-j}\ast\nu(x_{j-1})) \right)\right],
\end{align*}
where $\sup_{(x_j)_{j\le n}} \widehat{\e}^\omega_{(x_j)_{j\le n}}[r_j] \le  C_0$. For $1\le j\le m_n$ and $(x_j)_{j\le n}\in\mathscr{S}_n$, by \eqref{upbdp}, $p^\omega_{n-j}\ast\nu(x_{j-1})\le \frac{c^\omega_3}{\delta n}$ uniformly. %For $(1-\delta) n\le j\le n$, $\widehat{\e}^\omega_{(x_j)_{j\le n}}\left[ r_j\left(1 \wedge (l_j p^\omega_{n-j}\ast\nu(x_{j-1})) \right)\right]\le C$. 
Therefore by dominated convergence theorem, one gets that uniformly for $(x_j)_{j\le n}\in\mathscr{S}_n$ and for $1\le j\le m_n$, both $\widehat{\e}^\omega_{(x_j)_{j\le n}}\left[ 1 \wedge (l_j p^\omega_{n-j}\ast\nu(x_{j-1}))\right]$ and $\widehat{\e}^\omega_{(x_j)_{j\le n}}\left[ r_j\left(1 \wedge (l_j p^\omega_{n-j}\ast\nu(x_{j-1})) \right)\right]$ are $o_n^\omega(1)$. While for $m_n <j \le n$, we note that both $\widehat{\e}^\omega_{(x_j)_{j\le n}}\left[ 1 \wedge (l_j p^\omega_{n-j}\ast\nu(x_{j-1}))\right]$ and $\widehat{\e}^\omega_{(x_j)_{j\le n}}\left[ r_j\left(1 \wedge (l_j p^\omega_{n-j}\ast\nu(x_{j-1})) \right)\right]$ are bounded by $C_0+1$. Consequently,
\begin{align}\label{R1R*bd}
& \frac{2}{(C_0+1)n}   \widehat{\e}^{\omega, \textrm{2-spine}}\left[ | R^{I,*}_n - R^{I}_n| \right]  
%\le &\sum_{1\le j\le n(1-\delta )} \left( \widehat{\e}^\omega_{(x_j)_{j\le n}}[r_j] \widehat{\e}^\omega_{(x_j)_{j\le n}}\left[ 1 \wedge (l_j p^\omega_{n-j}\ast\nu(x_{j-1}))\right] + \widehat{\e}^\omega_{(x_j)_{j\le n}}\left[ r_j\left(1 \wedge (l_j p^\omega_{n-j}\ast\nu(x_{j-1})) \right)\right] \right) \\
 \le o_n^\omega(1) + 2C_2C_0\delta + 2C_2 \delta.
% &+ \frac{2C_2}{(C_0+1)n} \sum_{n-\delta n< j\le n} C_0 \widehat{\e}^\omega_{(x_j)_{j\le n}}\left[ 1 \wedge (l_j p^\omega_{n-j}\ast\nu(x_{j-1}))\right] +\widehat{\e}^\omega_{(x_j)_{j\le n}}\left[ r_j\left(1 \wedge (l_j p^\omega_{n-j}\ast\nu(x_{j-1})) \right)\right] .
\end{align}

On the other hand, note that 
\begin{align*}
&  \widehat{\e}^{\omega, \textrm{2-spine}}\left[ | R^{I}_n - R^{I\!I}_n| \right]  =    \widehat{\e}^{\omega, \textrm{2-spine}}\left[ |\sum_{j=1}^n(R_{n,j}^I - R_{n,j}^{I\!I})|  \right] \\
  =& \widehat{\e}^{\omega, \textrm{2-spine}}\left[ |\sum_{j=1}^n(R_{n,j}^I - R_{n,j}^{I\!I})| \ind{\tau \le m_n}  \right]  + \widehat{\e}^{\omega, \textrm{2-spine}}\left[ |\sum_{j=1}^n(R_{n,j}^I - R_{n,j}^{I\!I})| \ind{\tau > m_n}  \right] .
\end{align*}
Apparently,
\begin{align}\label{R1R2-1}
\widehat{\e}^{\omega, \textrm{2-spine}}\left[ |\sum_{j=1}^n(R_{n,j}^I - R_{n,j}^{I\!I})| \ind{\tau > m_n}  \right] \le& \sum_{m_n < j\le n} \widehat{\e}^{\omega, \textrm{2-spine}}[R_{n,j}^I +R_{n,j}^{I\!I} ] \\
\le & \sum_{m_n < j\le n} \widehat{\e}^{\omega, \textrm{2-spine}}[ r_j^I + r^{I\!I}_j ] \le 2\delta n C_0. \nonumber
\end{align}
Moreover,
\begin{align*}
& \widehat{\e}^{\omega, \textrm{2-spine}}\left[ |\sum_{j=1}^n(R_{n,j}^I - R_{n,j}^{I\!I})| \ind{\tau \le m_n}  \right] \le  \widehat{\e}^{\omega, \textrm{2-spine}}\left[ \sum_{j=1}^n(R_{n,j}^I + R_{n,j}^{I\!I})\ind{\tau \le m_n}  \right] \\
=& \sum_{(x_j), (y_j)\in\mathscr{S}_n}\ind{\cup_{j\le m_n}\{x_j \neq y_j\}} \widehat{\p}^\omega( (\Xi_j)=(x_j)_{j\le n}, (S_j)=(y_j)_{j\le n}) \widehat{\e}^{\omega, \textrm{2-spine}}_{(x_j)_{j\le n},(y_j)_{j\le n}}\left[ \sum_{j=1}^n(R_{n,j}^I + R_{n,j}^{I\!I})\right]\\
=& \sum_{(x_j), (y_j)\in\mathscr{S}_n}\ind{\cup_{j\le m_n}\{x_j \neq y_j\}} \widehat{\p}^\omega( (\Xi_j)=(x_j)_{j\le n}, (S_j)=(y_j)_{j\le n}) \sum_{j=1}^n\widehat{\e}^{\omega, \textrm{2-spine}}_{(x_j)_{j\le n},(y_j)_{j\le n}}\left[ (r_j^I + r_j^{I\!I})\right],
\end{align*}
where $\widehat{\e}^{\omega, \textrm{2-spine}}_{(x_j)_{j\le n},(y_j)_{j\le n}}\left[ (r_j^I + r_j^{I\!I})\right] \le 2\sup_{(x_j)_{j\le n}} \widehat{\e}^\omega_{(x_j)_{j\le n}}[r_j] \le 2C_0$. We thus deduce from \eqref{TV-coupling} that
\begin{align}\label{R1R2-2}
& \widehat{\e}^{\omega, \textrm{2-spine}}\left[ |\sum_{j=1}^n(R_{n,j}^I - R_{n,j}^{I\!I})| \ind{\tau \le m_n}  \right] \\
\le &  2  C_0n \widehat{\p}^\omega(\cup_{0\le j\le m_n}\{\Xi_j \neq S_j\}) \le 2C_0n \varepsilon_n(\delta, \omega). \nonumber
\end{align}
Combining \eqref{R1R2-1} and \eqref{R1R2-2} shows that
\begin{align}\label{R1R2bd}
\frac{2}{(C_0+1)n}\widehat{\e}^{\omega, \textrm{2-spine}}\left[ | R^{I}_n - R^{I\!I}_n| \right]  \le 4\delta + 4\varepsilon_n(\delta, \omega).
\end{align}
In view of \eqref{dWbd}, \eqref{R1R*bd}, and \eqref{R1R2bd}, we conclude that for any fixed $\delta\in(0,1)$, 
\begin{align*}
d_W( \mathscr{L}(\frac{Z_n}{ \e^\omega[Z_n \vert Z_n \ge 1] }, \p^\omega(\cdot\vert Z_n \ge1)), Exp(1) ) \le o_n(1) + o_n^\omega(1) + C_3\delta.
\end{align*}
Letting $n\to\infty$ and then $\delta\downarrow 0+$ yields that 
\[
\lim_{n\to\infty}d_W( \mathscr{L}(\frac{Z_n}{ \e^\omega[Z_n \vert Z_n \ge 1] }, \p^\omega(\cdot\vert Z_n \ge1)), Exp(1) ) =0.
\]
By \eqref{asymcondm}, it is immediate that
\[
\lim_{n\to\infty}d_W\left( \mathscr{L}(\frac{Z_n}{ \e^\omega[Z_n \vert Z_n \ge 1] }, \p^\omega(\cdot\vert Z_n \ge1)), \mathscr{L}(\frac{Z_n}{ \frac{p\sigma^2_\circ+(1-p)\sigma^2_\bullet}{2}n }, \p^\omega(\cdot\vert Z_n \ge1)) \right) =0.
\]
Theorem \ref{thm: Yaglom} follows then from triangle inequality for $d_W$.

\section{Proof of technical Lemmas}\label{lems}

We prove the technical propositions and lemmas here. 

\begin{proof}[Proof of Lemma \ref{spectralradius}]
 For \(z=(z_1,\dots,z_d)\in\mathbb{Z}^d\), define its parity by
\(
\operatorname{par}(z):=\sum_{j=1}^d z_j \pmod{2}.
\) 
For a simple random walk started at \(z\), parity is preserved at even times, i.e., \(\operatorname{par}(S_{2n})=\operatorname{par}(z)\) for all \(n\ge 0\). Fix \(x\in A\), and decompose \(A\subset\mathbb{Z}^d\) according to parity as \(A=A^{\text{even}}\cup A^{\text{odd}}\), where
\[
A^{\text{even}}:=\{y\in A:\operatorname{par}(y)=\operatorname{par}(x)\},
\qquad
A^{\text{odd}}:=A\setminus A^{\text{even}}.
\]

We say that $y$ and $z$
are neighbours if $y=z+e$ for some $e$ with $|e|_1=1$, and in this case write $y\sim z$.
Since \(\#A\ge 2\), we order the vertices of $A$ by listing first the points in \(A^{\text{even}}\), followed by the points in $A^{\text{odd}}$.
With respect to this ordering, the matrix representation of the operator $T_A$
has the block form
\[
T_A=
\begin{pmatrix}
0 & B\\
B^{\mathsf T} & 0
\end{pmatrix}.
\]

Since $T_A$ is symmetric, all its eigenvalues are real. Define the linear
operator $J:\mathbb{R}^A\to\mathbb{R}^A$ by
\[
(Jf)(y):=
\begin{cases}
\ \ f(y), & y\in A^{\text{even}},\\
- f(y), & y\in A^{\text{odd}}.
\end{cases}
\]
A direct computation shows that
\[
JT_A=-T_AJ.
\]
Consequently, if $T_A f=\mu f$ for some eigenvalue $\mu\in\mathbb{R}$, then
\[
T_A(Jf)=-\mu(Jf),
\]
which implies that $-\mu$ is also an eigenvalue of $T_A$. Therefore, the
spectrum of $T_A$ is symmetric about the origin. In particular, the spectral radius \(\rho_A\) of \(T_A\) coincides with its largest eigenvalue.

We now define the two-step operator restricted to the parity class $A^{\text{even}}$. Let
\[
T^{\text{even}}_A:\mathbb{R}^{A^{\text{even}}}\longrightarrow\mathbb{R}^{A^{\text{even}}},
\qquad
T^{\text{even}}_A := T_A^2\big|_{A^{\text{even}}}.
\]
With respect to the same basis, the matrix representation of $T^{\text{even}}_A$ is given by
\[
T^{\text{even}}_A = B B^{\mathsf T}.
\]
Moreover, for any \(y,z\in A^{\text{even}}\), we have
\begin{equation}\label{def-T-even}
(T^{\text{even}}_A)_{y,z}=\Q_y(S_2=z,\,\tau_A>2),
\qquad
\big((T^{\text{even}}_A)^n\big)_{y,z}=\Q_y(S_{2n}=z,\,\tau_A>2n).
\end{equation}

Since $A$ is finite, connected, and contains at least two points, the Markov chain induced by $T_A^{\text{even}}$ on $A^{\text{even}}$ is irreducible and aperiodic. By the Perron--Frobenius theorem, the largest eigenvalue of $T_A^{\text{even}}$ is strictly positive; we denote it by $\rho(T_A^{\text{even}})$. 
Furthermore,
\[
\lim_{l\to\infty}\frac{(T_A^{\text{even}})^l}{\rho(T_A^{\text{even}})^l}=K,
\]
where $K$ is a rank-one matrix with strictly positive entries. In particular, for each $z\in A^{\text{even}}$ there exists $k_z>0$ such that
\[
\lim_{l\to\infty}
\frac{\mathbf{Q}_x(S_{2l}=z,\ \tau_A>2l)}{\rho(T_A^{\text{even}})^l}
= k_z.
\]

Since $T_A^{\text{even}} = B B^{\mathsf T}$, all its eigenvalues are nonnegative. Let $\mu$ be an eigenvalue of $T_A^{\text{even}}$ with corresponding eigenvector $u$. Then one checks directly that
\[
\begin{cases}
    \begin{pmatrix} 
        u \\ \mu^{-1/2} B^{\mathsf T} u 
    \end{pmatrix}, & \text{if } \mu > 0, \\[2ex]
    \begin{pmatrix} 
        u \\ 0 
    \end{pmatrix}, & \text{if } \mu = 0,
\end{cases}
\]
is an eigenvector of $T_A$ with eigenvalue $\sqrt{\mu}$. Conversely, if $\varphi$ is an eigenvector of $T_A$ with eigenvalue $\lambda$, then its restriction $\varphi|_{A^{\text{even}}}$ is an eigenvector of $T_A^{\text{even}}$ with eigenvalue $\lambda^2$. In particular, we have
\begin{equation}\label{relation-rhoA-rhoeven}
    \rho(T_A^{\text{even}})=\rho_A^2.
\end{equation}

Moreover, since $A$ is finite, the simple random walk exits $A$ almost surely.
In particular, for every $z\in A^{\text{even}}$,
\[
\mathbf{Q}_x(\tau_A>2n,\ S_{2n}=z)\le \mathbf{Q}_x(\tau_A>2n)\xrightarrow[n\to\infty]{}0.
\]
Combined with the asymptotic relation
\(
\mathbf{Q}_x(\tau_A>2n,\, S_{2n}=z)\sim k_z\,\rho_A^{2n}
\)
with $k_z>0$, this implies
\[
0< \rho_A<1,
\]
which concludes the proof of the first part of Lemma \ref{spectralradius}.

Next consider $z\in A^{\text{odd}}$.
Using the one-step transition rule of the simple random walk, we obtain
\begin{align*}
\mathbf{Q}_x\bigl(S_{2n+1}=z,\ \tau_A>2n+1\bigr)
&= \sum_{y\sim z}
\mathbf{Q}_x\bigl(S_{2n}=y,S_{2n+1}=z,\ \tau_A>2n+1\bigr)\\
&= \frac{1}{2d}\sum_{y\sim z}
\mathbf{Q}_x\bigl(S_{2n}=y,\ \tau_A>2n\bigr).
\end{align*}
Since \(y \sim z\), it follows that \(y \in A^{\text{even}}\). Hence there exists \(k_y>0\) such that
\[
\mathbf{Q}_x(S_{2n}=y,\ \tau_A>2n)\sim k_y\,\rho_A^{2n}.
\]
Dividing by $\rho_A^{\,2n+1}$ and letting $n\to\infty$ yields
\[
\lim_{n\to\infty}
\frac{\mathbf{Q}_x(S_{2n+1}=z,\ \tau_A>2n+1)}{\rho_A^{\,2n+1}}
=\frac{1}{2d\,\rho_A}\sum_{y\sim z} k_y.
\]

Summing over parity classes gives
\[
\lim_{n\to\infty}
\frac{\mathbf{Q}_x(\tau_A>2n)}{\rho_A^{\,2n}}
=\sum_{y\in A^{\text{even}}} k_y\in(0,\infty),
\]
and
\[
\lim_{n\to\infty}
\frac{\mathbf{Q}_x(\tau_A>2n+1)}{\rho_A^{\,2n+1}}
=\frac{1}{2d\,\rho_A}
\sum_{z\in A^{\text{odd}}}
\sum_{y\sim z} k_y\in(0,\infty).
\]

This completes the proof of the second part of Lemma~\ref{spectralradius}. 
For the third part, we begin with the one-dimensional case.
Let
\[
A_n:=\{0,1,\dots,n\}\subset\mathbb{Z}.
\]

It can be verified by direct computation that the eigenvalues of $T_{A_n}$
are given by
\[
\mu_j=\cos\!\left(\frac{j\pi}{n+2}\right),
\qquad j=1,2,\dots,n+1,
\]
and that the corresponding eigenfunctions are
\[
\varphi_j(k)=\sin\!\left(\frac{j\pi(k+1)}{n+2}\right),
\qquad k=0,1,\dots,n.
\]
In particular, the largest eigenvalue is
\[
\rho_{A_n}=\cos\!\left(\frac{\pi}{n+2}\right),
\]
and the associated eigenfunction $\varphi_1$ is strictly positive on $A_n$. Moreover,
\[
\lim_{n\to\infty}\rho_{A_n}=1,
\]
which prove the third part of Lemma \ref{spectralradius} for dimension one. Now we extend this
conclusion to arbitrary dimension.

Let $R(A)$ denote the maximal radius of an $\ell^1$-ball contained in $A$.
Without loss of generality, assume that the largest such ball is
$B_R(y)\subset A$. Since $A$ is finite and connected, for any $x\in A$ there
exists $m\in\mathbb{N}$ and an admissible path $(y_0=x, y_1,\cdots, y_m=y)$ inside $A$ such that
\[
\mathbf{Q}_x((S_j)_{j\le m}=(y_j)_{j\le m})>0.
\]
For any $n>m$, we have
\begin{align*}
\mathbf{Q}_x(\tau_A>n)
%&=\sum_{z\in A}\mathbf{Q}_x(\tau_A>n,\ S_m=z)
\;\ge\;\mathbf{Q}_x(\tau_A>n,\ (S_j)_{j\le m}=(y_j)_{j\le m}).
\end{align*}
By the Markov property, and \(\Q_y(\tau_{B_R(y)}\le \tau_A)=1\),
\begin{align*}
\mathbf{Q}_x(\tau_A>n)
%&\ge \mathbf{Q}_x(\tau_A>n\mid S_m=y)\,\mathbf{Q}_x((S_j)_{j\le m}=(y_j)_{j\le m})\\
&\ge \mathbf{Q}_y(\tau_A>n-m)\,\mathbf{Q}_x((S_j)_{j\le m}=(y_j)_{j\le m})\\
&\ge \mathbf{Q}_y(\tau_{B_R(y)}>n-m)\,\mathbf{Q}_x((S_j)_{j\le m}=(y_j)_{j\le m}).
\end{align*}

From the first part of the lemma we know that, for every $x\in A$,
\[
\rho_A
=\lim_{j\to\infty}
\exp\!\left(\frac{1}{j}\log \mathbf{Q}_x(\tau_A>j)\right).
\]
Combining this with the above lower bound yields
\begin{align*}
\rho_A
&\ge
\lim_{j\to\infty}
\exp\!\left(
\frac{\log \mathbf{Q}_y(\tau_{B_R(y)}>j-m)
+\log \mathbf{Q}_x((S_j)_{j\le m}=(y_j)_{j\le m})}{j}
\right)
=\rho_{B_R(0)},
\end{align*}
where we used translation invariance in the last step. Hence it suffices to
prove that
\[
\lim_{R\to\infty}\rho_{B_R(0)}=1.
\]

To this end, consider $d$ independent one-dimensional simple random walks
$(\widetilde S^{(1)},\dots,\widetilde S^{(d)})$, each started from the origin.
Let $(U_k)_{k\ge1}$ be an i.i.d.\ sequence, independent of the walks, uniformly
distributed on $\{1,\dots,d\}$, and denote by $N_n(i)$ the number of indices
$k\le n$ such that $U_k=i$. Then the $d$-dimensional simple random walk
$(S_n)_{n\ge0}$ satisfies the distributional identity
\[
S_n\stackrel{d}{=}
\bigl(\widetilde S^{(1)}_{N_n(1)},\dots,\widetilde S^{(d)}_{N_n(d)}\bigr).
\]
With a slight abuse of notation, we use $\mathbf{Q}_0$ to denote the joint law
of the $d$ independent one-dimensional simple random walks
$(\widetilde S^{(k)})_{1\le k\le d}$. Then
\[
\mathbf{Q}_0(\tau_{B_R(0)}>n)
=
\mathbf{Q}_0\!\left(\forall\, j\le n,\ |S_j|_1<R\right)
=
\mathbf{Q}_0\!\left(\forall\, j\le n,\ \sum_{k=1}^d
\bigl|\widetilde S^{(k)}_{N_j(k)}\bigr|_1<R\right).
\]
Moreover,
\begin{align*}
\mathbf{Q}_0\!\left(\forall\, j\le n,\ \sum_{k=1}^d
\bigl|\widetilde S^{(k)}_{N_j(k)}\bigr|_1<R\right)
&\ge
\mathbf{Q}_0\!\left(
\bigcap_{k=1}^d
\left\{\forall\, j\le n,\ 
\bigl|\widetilde S^{(k)}_{N_j(k)}\bigr|_1<\frac{R}{d}\right\}
\right)\\
&\ge
\mathbf{Q}_0\!\left(
\bigcap_{k=1}^d
\left\{\forall\, j\le n,\ 
\bigl|\widetilde S^{(k)}_{j}\bigr|_1<\frac{R}{d}\right\}
\right).
\end{align*}
Since the processes $(\widetilde S^{(k)})_{1\le k\le d}$ are independent, it
follows that
\begin{align*}
\mathbf{Q}_0\!\left(
\bigcap_{k=1}^d
\left\{\forall\, j\le n,\ 
\bigl|\widetilde S^{(k)}_{j}\bigr|_1<\frac{R}{d}\right\}
\right)
&=
\left[
\mathbf{Q}_0\!\left(
\forall\, j\le n,\ 
\bigl|\widetilde S^{(1)}_{j}\bigr|_1<\frac{R}{d}
\right)
\right]^d\\
&=
\mathbf{Q}_0\!\left(
\tau_{\{-[R/d]+1,\dots,[R/d]-1\}}>n
\right)^d.
\end{align*}
Combining the above inequalities yields
\[
\mathbf{Q}_0(\tau_{B_R(0)}>n)
\ge
\mathbf{Q}_0\!\left(
\tau_{\{-[R/d]+1,\dots,[R/d]-1\}}>n
\right)^d.
\]

By the one-dimensional result established earlier,
\[
\lim_{R\to\infty}
\rho_{\{-[R/d]+1,\dots,[R/d]-1\}}=1.
\]
This implies $\lim_{R\to\infty}\rho_{B_R(0)}=1$, and therefore
$\rho_A\to1$ as $R(A)\to\infty$, which completes the proof of the third
assertion of Lemma~\ref{spectralradius}.
\end{proof}

\begin{proof}[Proof of Proposition \ref{brw+killing}]

We begin by recalling basic facts about multi-type branching processes. A \(D\)-type branching process
\[
N_n=(N_n(1),\, N_n(2),\,\dots,\, N_n(D))
\]
is a Markov chain in state space \((\N_{\ge 0})^D\) with reproduction law \(\{\mathbf{P}^{(k)}:\,k=1,\dots,D\}\). In other words, for any \(\mathbf{i}\in (\N_{\ge 0})^D\),
\[
\mathbf{P}^{(k)}(\mathbf{i}) := \p\left(N_1=\mathbf{i} \Big| N_0=\Big(0, \dots, 0,\underset{\substack{\uparrow \\ k\text{-th}}}{1}, 0,\dots, 0\Big)\right).
\]
The initial state \(N_0 \neq \mathbf{0}\) is a random vector in \((\N_{\ge 0})^D\), and the evolution is given by
\[
N_{n+1}=\sum_{k=1}^D\sum_{j=1}^{N_n(k)}\xi_{j,n}^{(k)},
\]
where for each fixed type \(k\), the random vectors \(\xi_{j,n}^{(k)}\) are i.i.d. with law \(\P^{(k)}\) on \((\N_{\ge 0})^D\), and the collection
\[
\left\{\xi_{j,n}^{(k)}:\, k\in \{1,\dots,D\},\, j\in \N_{\ge 1},\, n\in\N_{\ge 1}\right\}
\]
is mutually independent. The mean matrix \(M\) of the process \((N_n)_{n\in\N}\) is the \(D\times D\) matrix defined by
\[
M_{k,j}=\sum_{\mathbf{i}\in\N_{\ge 0}^D} \mathbf{i}(j)\, \P^{(k)}(\mathbf{i}),
\]
where \(\mathbf{i}(j)\) denotes the \(j\)-th coordinate of \(\mathbf{i}\in\N_{\ge 0}^D\). It is well known (see, e.g., \cite[Chapter V, Section 3, Theorem 2]{AN}) that, under the irreducibility assumption that there exists \(k\in\N_{\ge 1}\) such that all entries of \(M^k\) are strictly positive, the survival probability
\[
\p(\forall n\in\N_{\ge 1},\, N_n\neq \mathbf{0})
=\p\!\left(\forall n\in\N_{\ge 1},\, \sum_{k=1}^D N_n(k)\ge 1\right)
\]
is positive if and only if the the largest eigenvalue \(\rho(M)\) satisfies \(\rho(M)>1\). For \(z=(z_1,\dots,z_d)\in\mathbb{Z}^d\), define its parity by
\(
\operatorname{par}(z):=\sum_{j=1}^d z_j \pmod{2}.
\) 
Fix \(x\in A\), and decompose \(A\subset\mathbb{Z}^d\) according to parity as \(A=A^{\text{even}}\cup A^{\text{odd}}\).%, where
% \[
% A_{\text{even}}:=\{y\in A:\operatorname{par}(y)=\operatorname{par}(x)\},
% \qquad
% A_{\text{odd}}:=A\setminus A_{\text{even}}.
% \]

Without loss of generality, we can assume that $\sum_k kp_k<\infty$. Otherwise, we could take a truncated offspring law with finite mean. Let \(Z_n(y)\) denote the number of particles at site \(y\in A\) at time \(n\). Then
\[
\Big(\{Z_{2n}(y):\, y\in A^{\text{even}}\}\Big)_{n=0,1,\dots}
\]
forms a multi-type branching process. Conditioning on the first generation yields that for all \(y,z\in A^{\text{even}}\),
\[
\e_{\p^{(A)}_z}[Z_{2}(y)]
= \left(\sum_{k=0}^\infty k p_k\right)^2
  \Q_z\big(S_2=y,\, \tau_A>2\big).
\]

In other words, using the definition of \(T_A^{\text{even}}\) in \eqref{def-T-even}, the mean matrix of the process \(\{Z_{2n}(y):\, y\in A^{\text{even}}\}\) is
\[
M^{\text{even}}=\left(\sum_{k} k p_k\right)^2\, T_A^{\text{even}}.
\]
Moreover, all entries of \((M^{\text{even}})^{\#A}\) are strictly positive, and by \eqref{relation-rhoA-rhoeven}, the largest eigenvalue of \(M^{\text{even}}\) is given by
\[
\rho(M^{\text{even}})=\left(\sum_{k} k p_k\right)^2 \rho_A^2.
\]
Consequently, if \(\left(\sum_k k p_k\right)\rho_A>1\), then
\[
\p(\forall n\in\N_{\ge 1},\, Z_n\ge 1)
=\p(\forall n\in\N_{\ge 1},\, Z_{2n}\ge 1)>0,
\]
whereas if \(\left(\sum_k k p_k\right)\rho_A\le 1\), then
\[
\p(\forall n\in\N_{\ge 1},\, Z_n\ge 1)
=\p(\forall n\in\N_{\ge 1},\, Z_{2n}\ge 1)=0.
\]
\end{proof}

\begin{proof}[Proof of Lemma \ref{percolation}]
Assume that $\theta_{\Z^d}(p) > 0$. Fix an arbitrary integer $R \ge 1$. Let $A_R(x)$ denote the event that the ball $B_x(R)$ centered at $x \in \Z^d$ is entirely open and that the open cluster containing $x$ is infinite. The event $A_R(x)$ can be written as the intersection of two increasing events: $E_1 = \{B_x(R)\ \text{is fully open}\}$ and $E_2 = \{\#\mathcal{C}(x) = \infty\}$. By the FKG inequality, we have
\[
\P_p(A_R(x)) \ge \P_p(E_1)\P_p(E_2) = p^{|B_x(R)|}\,\theta_{\Z^d}(p).
\]
Since $p > 0$ and $\theta_{\Z^d}(p) > 0$, it follows that $\P_p(A_R(x)) > 0$, and hence
\[
\P_p\Bigl(\bigcup_{x \in \Z^d} A_R(x)\Bigr) > 0.
\]

By Kolmogorov's $0$--$1$ law, this probability is in fact equal to $1$. Moreover, by the uniqueness of the infinite cluster (see \cite{BK}), there exists a unique infinite open cluster, denoted by $\mathcal{C}_\infty$, almost surely. Consequently,
\[
\P_p\bigl(\exists\, x \in \Z^d \text{ such that } B_x(R) \subset \mathcal{C}_\infty \bigr) = 1.
\]

Since this holds for every fixed $R \ge 1$, taking the countable intersection over all $R$ yields
\[
\P_p\bigl(\forall R \ge 1,\ \exists x \in \Z^d \text{ such that } B_x(R) \subset \mathcal{C}_\infty \bigr) = 1.
\]
Finally, conditioning on the positive-probability event $\{\#\mathcal{C}(0) = \infty\}$ and noting that $\mathcal{C}(0) = \mathcal{C}_\infty$ under this conditioning, we conclude that
\[
\P_p\left( \forall R \ge 1,\ \exists x \in \mathcal{C}(0) \text{ such that } B_x(R) \subset \mathcal{C}(0) \,\middle|\, \#\mathcal{C}(0) = \infty \right) = 1.
\]
This completes the proof.
\end{proof}
\begin{proof}[Proof of Lemma \ref{couple-spine}]
We check that for any fixed $\delta\in(0,1)$, $\P_p$-a.s., as $n\to\infty$,
\begin{equation}\label{TVcvg}
 \sum_{(x_j)_{j\le m_n} \in \mathscr{S}_{m_n}} | \widehat{\p}^\omega( (\Xi_j^{(n)})_{j\le m_n} = (x_j)_{j\le m_n} ) - \P( (S_j)_{j\le m_n} = (x_j)_{j\le m_n}) | \to 0.
\end{equation}

First, recall that
\[
\widehat{\p}^\omega( (\Xi_j^{(n)})_{j\le m_n} \in\cdot ) = \widehat{\p}^\omega( (S_{\zeta_j})_{j\le m_n} \in\cdot \vert \cap_{j=1}^n\{L_{n,j}=0\} ).
\]
Now let
\[
A_n^\delta:=\cap_{1\le j\le m_n}\{ L_{n,j}=0\},
\]
and introduce a new random vector $(\Xi_j^{(n,\delta)})_{0\le j\le m_n}$ satisfying that 
\begin{align}
\widehat{\p}^\omega( (\Xi_j^{(n,\delta)})_{0\le j\le m_n} \in \cdot ) = \widehat{\p}^\omega( (S_{\zeta_j})_{j\le m_n} \in\cdot\vert A_n^\delta).
\end{align}
Then \eqref{TVcvg} follows from the following two estimates:
\begin{align}\label{TVcvg-1}
\Sigma_\eqref{TVcvg-1}:=
&  \sum_{(x_j)_{j\le m_n} \in \mathscr{S}_{m_n}} | \widehat{\p}^\omega( (\Xi_j^{(n,\delta)})_{j\le m_n} = (x_j)_{j\le m_n} ) - \P( (S_j)_{j\le m_n} = (x_j)_{j\le m_n}) | \\
=& o_n^\omega(1),\nonumber
\end{align}
\begin{align}\label{TVcvg-2}
\Sigma_\eqref{TVcvg-2}:=&\sum_{(x_j)_{j\le m_n} \in \mathscr{S}_{m_n}} | \widehat{\p}^\omega( (\Xi_j^{(n)})_{j\le m_n} = (x_j)_{j\le m_n} ) - \widehat{\p}^\omega( (\Xi_j^{(n,\delta)})_{j\le m_n} = (x_j)_{j\le m_n} ) | \\
=& o_n^\omega(1).\nonumber
\end{align}

\paragraph{Step 1: Proof of \eqref{TVcvg-1}.} Observe that 
\begin{align*}
\Sigma_\eqref{TVcvg-1}=& \sum_{(x_j)_{j\le m_n} \in \mathscr{S}_{m_n}} |  \frac{ \widehat{\p}^\omega( (S_{\zeta_j})_{j\le m_n} = (x_j)_{j\le m_n}; A_n^\delta) }{ \widehat{\p}^\omega(A_n^\delta) } -   \widehat{\p}^\omega( (S_{\zeta_j})_{j\le m_n} = (x_j)_{j\le m_n}) |\\
=&  \sum_{(x_j)_{j\le m_n} \in \mathscr{S}_{m_n}}  \widehat{\p}^\omega( (S_{\zeta_j})_{j\le m_n} = (x_j)_{j\le m_n}) |\frac{ \widehat{\p}^\omega(A_n^\delta \vert  (S_{\zeta_j})_{j\le m_n} = (x_j)_{j\le m_n} ) }{ \widehat{\p}^\omega(A_n^\delta) }-1| \\
=& \widehat{\e}^\omega\left[ \bigg|\frac{ \widehat{\e}^\omega(\ind{A_n^\delta} \vert  (S_{\zeta_j})_{j\le m_n} ) }{ \widehat{\p}^\omega(A_n^\delta) }-1\bigg| \right].
\end{align*}
We view $\frac{ \widehat{\e}^\omega(\ind{A_n^\delta} \vert  (S_{\zeta_j})_{j\le m_n} ) }{ \widehat{\p}^\omega(A_n^\delta) }$ as a random variable with mean $1$. Then Cauchy-Schwartz inequality shows that
\begin{align*}
\Sigma_\eqref{TVcvg-1} \le & \sqrt{ \widehat{\e}^\omega\left[ \bigg(\frac{ \widehat{\e}^\omega(\ind{A_n^\delta} \vert  (S_{\zeta_j})_{j\le m_n} ) }{ \widehat{\p}^\omega(A_n^\delta) }-1\bigg)^2 \right] } = \sqrt{ Var\left( \frac{ \widehat{\e}^\omega(\ind{A_n^\delta} \vert  (S_{\zeta_j})_{j\le m_n} ) }{ \widehat{\p}^\omega(A_n^\delta) } \right)}.
\end{align*}
A classical inequality on conditional variance says that if there are two sigma-algebras $\mathcal{G}\subset\mathcal{H}$ on the probability space, then 
\[
\operatorname{Var}\!\left(\widehat\e^\omega[X\mid\mathcal G]\right)
 \le
 \operatorname{Var}\!\left(\widehat\e^\omega[X\mid\mathcal H]\right),
 \qquad \mathcal G\subset\mathcal H.
\]
See for instance \cite[Section 5.1]{Durrett}.  Now we take $\mathcal{G}= \sigma(  (S_{\zeta_j})_{j\le m_n}  )$ and $\mathcal{H}= \sigma(  (\zeta_j, S_{\zeta_j})_{j\le m_n}; \{ (l_j, r_j) \}_{j\le m_n})$. Consequently,
\begin{align}\label{TVcvg-1bd}
\Sigma_\eqref{TVcvg-1} \le \sqrt{ Var\left( \frac{  \widehat{\e}^\omega(\ind{A_n^\delta} \vert  \mathcal{H} ) }{ \widehat{\p}^\omega(A_n^\delta) } \right)}.
\end{align}
Observe that
\begin{align*}
 \widehat{\e}^\omega(\ind{A_n^\delta} \vert  \mathcal{H} ) = & \widehat{\e}^\omega\left[ \prod_{1\le j\le m_n } \ind{L_{n,j}=0 } \vert \mathcal{H} \right] =  \prod_{1\le j\le m_n} \left( 1- p^\omega_{n-j}\ast\nu (S_{\zeta_{j-1}}) \right)^{l_j} \\
 =& \exp\left\{ \sum_{1\le j\le m_n} l_j \log \left( 1- p^\omega_{n-j}\ast\nu (S_{\zeta_{j-1}}) \right) \right\}.
\end{align*}
Theorem \eqref{thm: UKprobab} shows that uniformly in $1\le j\le m_n$, 
\[
 p^\omega_{n-j}\ast\nu (S_{\zeta_{j-1}}) = \frac{1+ o^\omega_n(1)}{ \frac{\overline{\sigma}^2}{2}(n-j) }.
\]
As $\log(1-h) = - h(1+o_h(1))$ for $0< h\ll 1$, we therefore obtain that
\begin{align*}
 \widehat{\e}^\omega(\ind{A_n^\delta} \vert  \mathcal{H} ) = &  \exp\left\{ -(1+o^\omega_n(1)) \sum_{1\le j\le m_n} l_j \frac{1}{ \frac{\overline{\sigma}^2}{2}(n-j) } \right\}.
\end{align*}
Here we claim that for $\P_p$-a.s. environment $\omega$, 
\begin{equation}\label{reweightedsum}
 \sum_{1\le j\le m_n}\frac{l_j}{n-j} \xrightarrow[n\to\infty]{\widehat{\p}^\omega-a.s.} (-\log \delta)\frac{\overline{\sigma}^2}{2} .
\end{equation}
By admitting this claim, we get that for $\P_p$-a.s. environment $\omega$,
\[
 \widehat{\e}^\omega( \ind{A_n^\delta} \vert  \mathcal{H} ) \xrightarrow[n\to\infty]{\widehat{\p}^\omega-a.s.} \delta.
\]
As $\ind{A_n^\delta}\le 1$, dominated convergence theorem shows that $\P_p$-a.s.,
\begin{equation}\label{pAn}
 \widehat{\p}^\omega(A_n^\delta) = \widehat{\e}^\omega[ \widehat{\e}^\omega( \ind{A_n^\delta} \vert  \mathcal{H} ) ] \to \delta.
\end{equation}
Going back to \eqref{TVcvg-1bd}, as $\limsup_n|  \frac{  \widehat{\e}^\omega(\ind{A_n^\delta} \vert  \mathcal{H} ) }{ \widehat{\p}^\omega(A_n^\delta) } -1 |\le \frac{1}{ \delta }+1<\infty$, again by dominated convergence theorem, we deduce that for $\P_p$-a.s. environment $\omega$,
\begin{align*}
\Sigma_\eqref{TVcvg-1} = o_n^\omega(1).
\end{align*}

\paragraph{Step 2: proof of \eqref{TVcvg-2}.} It is immediate that
\begin{align*}
\Sigma_\eqref{TVcvg-2}=&\sum_{(x_j)_{j\le m_n} \in \mathscr{S}_{m_n}} \widehat{\p}^\omega( (\Xi_j^{(n,\delta)})_{j\le m_n} = (x_j)_{j\le m_n} ) \bigg| \frac{ \widehat{\p}^\omega( (\Xi_j^{(n)})_{j\le m_n} = (x_j)_{j\le m_n} )}{ \widehat{\p}^\omega( (\Xi_j^{(n,\delta)})_{j\le m_n} = (x_j)_{j\le m_n} ) } - 1 \bigg| \\
\le & \sup_{ (x_j)_{j\le m_n} \in \mathscr{S}_{m_n}} \bigg| \frac{ \widehat{\p}^\omega( (\Xi_j^{(n)})_{j\le m_n} = (x_j)_{j\le m_n} )}{ \widehat{\p}^\omega( (\Xi_j^{(n,\delta)})_{j\le m_n} = (x_j)_{j\le m_n} ) } - 1 \bigg|.
\end{align*}
Observe that
\begin{align*}
 &\frac{ \widehat{\p}^\omega( (\Xi_j^{(n)})_{j\le m_n} = (x_j)_{j\le m_n} )}{ \widehat{\p}^\omega( (\Xi_j^{(n,\delta)})_{j\le m_n} = (x_j)_{j\le m_n} ) }
 = \frac{ \widehat{\p}^\omega( (S_{\zeta_j})_{j\le m_n}= (x_j)_{j\le m_n} \vert \cap_{j=1}^n\{L_{n,j}=0\} ) }{ \widehat{\p}^\omega( (S_{\zeta_j})_{j\le m_n}= (x_j)_{j\le m_n} \vert A_n^\delta  ) } \\
 =&\frac{  \widehat{\p}^\omega( (S_{\zeta_j})_{j\le m_n}= (x_j)_{j\le m_n}; \cap_{j=1}^n\{L_{n,j}=0\} ) \widehat{\p}^\omega(A_n^\delta) }{ \widehat{\p}^\omega( (S_{\zeta_j})_{j\le m_n}= (x_j)_{j\le m_n} ;A_n^\delta )\widehat{\p}^\omega( \cap_{j=1}^n\{L_{n,j}=0\} )}
\end{align*}
which by the Markov property at time $m_n$ equals to 
\begin{align*}
=& \frac{  \widehat{\p}^\omega( (S_{\zeta_j})_{j\le m_n}= (x_j)_{j\le m_n};A_n^\delta ) \widehat{\p}^\omega_{x_{m_n}}(L_{\delta n}=0) \widehat{\p}^\omega(A_n^\delta) }{ \widehat{\p}^\omega( (S_{\zeta_j})_{j\le m_n}= (x_j)_{j\le m_n} ;A_n^\delta )\widehat{\e}^\omega[\ind{A_n^\delta} \widehat{\p}^\omega_{S_{\zeta_{m_n}}}( L_{\delta n} =0 )] } \\
= & \frac{  \widehat{\p}^\omega_{x_{m_n}}(L_{\delta n}=0) \widehat{\p}^\omega(A_n^\delta) }{ \widehat{\e}^\omega[\ind{A_n^\delta} \widehat{\p}^\omega_{S_{\zeta_{m_n}}}( L_{\delta n} =0 )]  }= \frac{ p^\omega_{\delta n}(x_{m_n}) \widehat{\p}^\omega(A_n^\delta)  }{  \widehat{\e}^\omega[\ind{A_n^\delta} p^\omega_{\delta n}(S_{\zeta_{m_n}})]  },
\end{align*}
where the last equality comes from \eqref{probab-spine}. Note that $|x_{m_n}|_1\le n$ and $|S_{\zeta_{m_n}} |_1\le n$, so the uniform convergence in Theorem \ref{thm: UKprobab} implies that
\[
p^\omega_{\delta n}(x_{m_n})  = \frac{1+o_n^\omega(1)}{\delta n}\frac{2}{\overline{\sigma}^2}, \textrm{ and } p^\omega_{\delta n}(S_{\zeta_{m_n}})=  \frac{1+o_n^\omega(1)}{\delta n}\frac{2}{\overline{\sigma}^2}.
\]
As a result, uniformly for $(x_j)_{j\le m_n} \in \mathscr{S}_{m_n}$,
\begin{align*}
&\bigg| \frac{ \widehat{\p}^\omega( (\Xi_j^{(n)})_{j\le m_n} = (x_j)_{j\le m_n} )}{ \widehat{\p}^\omega( (\Xi_j^{(n,\delta)})_{j\le m_n} =  (x_j)_{j\le m_n} ) } - 1\bigg|= \bigg| \frac{ (1+o_n^\omega(1))\widehat{\p}^\omega(A_n^\delta)  }{ \widehat{\e}^\omega[\ind{A_n^\delta}(1+o_n^\omega(1))] }-1 \bigg| \\
=& \frac{ o_n^\omega(1) }{\widehat{\p}^\omega(A_n^\delta) + o_n^\omega(1) }.
\end{align*}
It then follows from \eqref{pAn} that
\begin{align*}
\Sigma_\eqref{TVcvg-2} = \frac{ o_n^\omega(1) }{\widehat{\p}^\omega(A_n^\delta) + o_n^\omega(1) } = o_n^\omega(1),
\end{align*}
which completes the proof of \eqref{TVcvg-2}.

It remains to prove \eqref{reweightedsum}. In fact, it follows from Lemma \ref{LLNlr} and the following fact: for a real sequence $\{y_n\}_{n\ge 1}$,  % satisfies that $\frac1n \sum_{j=1}^n y_j \to a\in\R$, then
\begin{equation}\label{reweightedy}
\frac1n \sum_{j=1}^n y_j \to a\in\R \Longrightarrow \sum_{1\le j\le m_n}\frac{y_j}{n-j} \to (-\log\delta)a.
\end{equation}
Set $m=m_n$ and $s_k:=\sum_{j=1}^k y_j$ with $s_0:=0$. Then one sees that
\begin{align*}
\sum_{1\le j\le m_n}\frac{y_j}{n-j} = & \sum_{1\le j\le m_n}\frac{s_j - s_{j-1}}{n-j} = \sum_{j=1}^{m}\frac{s_j}{n-j} - \sum_{j=1}^{m-1} \frac{s_j}{n-j-1} \\
= & \frac{ s_m}{n-m} + \sum_{j=1}^{m-1} [\frac{s_j}{n-j}- \frac{s_j}{n-j-1}]\\
=& \frac{ s_m}{n-m} - \sum_{j=1}^{m-1} \frac{aj}{(n-j)(n-j-1)} - \sum_{j=1}^{m-1} (\frac{s_j}{j}-a)\frac{j}{(n-j)(n-j-1)}.
\end{align*}
On the one hand, it is apparent that
\[
 \frac{ s_m}{n-m} \to \frac{1-\delta}{\delta} a, \textrm{ and } \sum_{j=1}^{m-1} \frac{aj}{(n-j)(n-j-1)} \to \int_0^{1-\delta} \frac{audu}{(1-u)^2} =[ \log\delta + \frac{1-\delta}{\delta}]a.
\]
On the other hand, as $s_n/n\to a$, $C_4:=\sup_{j\ge 1}|\frac{s_j}{j}-a |<\infty$. We note that
\begin{align*}
|\sum_{j=1}^{m-1} (\frac{s_j}{j}-a)\frac{j}{(n-j)(n-j-1)}| \le & \sum_{j=1}^{m-1} | \frac{s_j}{j}-a| \frac{j}{(n-j)(n-j-1)} \\
\le & \sum_{j=1}^{\sqrt{n}} C_4 \frac{j}{(\delta n)^2} + \sum_{\sqrt{n} \le j\le m_n}  | \frac{s_j}{j}-a| \frac{n}{(\delta n)^2}\\
\le & C_4 \frac{1}{\delta^2 n} + o_n(1) \frac{1}{\delta^2} = o_n(1).
\end{align*}
We thus conclude \eqref{reweightedy} which suffices to show \eqref{reweightedsum}.
\end{proof}

\section*{Acknowledgment.} 
We would like to thank Yuval Peres and Zhenyao Sun for the enlightening discussions.

X. C. is supported by National Key R\&D Program of China (No. 2022YFA1006500) and the National Natural Science Foundation of China (Grant No. 12571148) and by the Fundamental Research Funds for the Central Universities (No. 310432104).. C. G. is supported by National Key R\&D Program of China (No. 2023YFA1010400) and National Natural Science Foundation of China (Nos. 12595280, 12595284).

%\section{Weak convergence conditionally on $\tM\le -x$}
%
%\section{Weak convergence conditionally on $W_\infty \ge x$}

\bibliographystyle{plain}
\bibliography{ref.bib}

\end{document}